\documentclass[smallextended]{svjour3}       
\smartqed
\RequirePackage[OT1]{fontenc}
\RequirePackage[numbers]{natbib}
\RequirePackage[colorlinks,citecolor=blue,urlcolor=blue]{hyperref}

\usepackage[tikz]{bclogo}
\usepackage{algorithm2e}
\usepackage{amssymb, amsmath, natbib, graphicx, euscript, mathrsfs, enumerate, empheq,color, lscape}

\newcommand{\pp}{\mathscr{P}}
\newcommand{\eqd}{\stackrel{Law}{=}}

\newcommand{\psim}{  \psi_j^{(m)}}
\renewcommand{\S}{S}

\newcommand{\PP}{{\mathcal P}}

\renewcommand{\H}{{\mathscr H}}

\newcommand{\wM}{\widetilde{\M}}

\newcommand{\E}{{\mathbb E}}

\newcommand{\hp}{\hat{p}}
\newcommand{\too}{\rightsquigarrow}
\newcommand{\C}{\mathcal{C}}
\newcommand{\M}{\mathcal{M}}

\def\R{\mathbb{R}}
\def\N{\mathbb{N}}

\newcommand{\eps}{\varepsilon}

\newcommand{\Var}{\operatorname{Var}}
\newcommand{\cov}{\operatorname{cov}}

\newcommand{\argmin}{\operatornamewithlimits{arg\,min}}
\newcommand{\argmax}{\operatornamewithlimits{arg\,max}}

\renewcommand{\P}{{\mathbb P}}
\renewcommand{\L}{\mathscr{L}}

\renewcommand{\u}{\breve{u}}

\newcommand{\B}{\mathcal{B}}

\newcommand{\W}{\mathcal{W}}

\renewcommand{\kappa}{\varkappa}

\newcommand{\p}{\Upsilon}

\newcommand{\ccc}{\mathfrak{c}}
\newcommand{\DD}{\mathcal{D}}
\newcommand{\DDD}{\DD}
\newcommand{\RR}{\mathcal{R}}
\newcommand{\I}{{\mathbb I}}

\newcommand{\F}{\mathcal{F}}
\renewcommand{\W}{\mathscr{W}}

\newtheorem{innercustomthm}{Lemma}
\newenvironment{customthm}[1]
  {\renewcommand\theinnercustomthm{#1}\innercustomthm}
  {\endinnercustomthm}

\begin{document}
\title{Extremes of Gaussian non-stationary processes and maximal deviation of projection density estimates
\thanks{The study has been funded by the Russian Academic Excellence Project '5-100'}}

\titlerunning{Extremes of Gaussian non-stationary processes}

\author{Valentin Konakov$^1$\\
Vladimir Panov$^1 (corresponding \; author)$\\Vladimir Piterbarg$^{1,2}$ 
}

\authorrunning{V.Konakov, V.Panov, and V.Piterbarg} 

\institute{1-International Laboratory of Stochastic Analysis and its Applications\\National Research University Higher School of Economics\\
             Pokrovsky boulevard 11, 109028 Moscow, Russia\newline
             \newline
             2 - Faculty of Mechanics and Mathematics \\
             Lomonosov Moscow State University\\
             GSP-1 Leninskie Gory, 119991 Moscow, Russia\newline
             \\                         
	\email{vkonakov@hse.ru         \and  vpanov@hse.ru \and piter@mech.math.msu.su
	}
}

\date{Received: \today}
\maketitle

\begin{abstract}
In this paper, we consider the distribution of the supremum of non-stationary Gaussian processes, and 
 present a new theoretical result on the asymptotic behaviour of this distribution. Unlike previously known facts in this field, our main theorem yields the asymptotic representation of the corresponding distribution function with exponentially decaying remainder term.  This result can be efficiently used for studying the projection density estimates, based, for instance, on Legendre polynomials. More precisely, we construct the sequence of accompanying laws, which approximates the distribution of maximal deviation of the considered estimates with polynomial rate.   Moreover, we  construct the confidence bands for densities, which are honest at polynomial rate to a broad class of densities.  

\end{abstract}

\setcounter{tocdepth}{3}
\setcounter{secnumdepth}{3}
\newpage

\section{Introduction}
\label{intro}
Consider the probability 
\begin{eqnarray} \label{question}
P_{u}(X,M)=\P\bigl\{  \sup_{t\in M}X (t)\geq u \bigr\} 
\end{eqnarray}
for a Gaussian process \(X=\bigl(X(t)\bigr)_{t \geq 0}\), and a set \(M \subset \R\). The asymptotic behaviour of \(P_{u}(X,M)\) as \(u \to \infty,\) is a classical question in the extreme value theory for Gaussian processes, see, e.g.,  the book by Adler and Taylor (\citeyear{AT}).

 Not surprisingly, the answer drastically depends on the properties of the covariance function of \(X.\) For stationary processes, the question is discussed in details, for instance, in the monographs by Piterbarg (\citeyear{Piterbarg}, \citeyear{Piterbarg20}). Nevertheless, for nonstationary processes, only the main term in the asymptotics of~\eqref{question} is known. In fact, almost all  results of this type are based on Pickand's and Berman's methods, which do not allow to make any conclusions on the behaviour of the second and further terms, see Piterbarg and Prisiazhniuk (\citeyear{PP}), Hashorva and H\"usler (\citeyear{HH2000}), H\"usler and Piterbarg (\citeyear{HP}), Bai, D\c{e}bicki, Hashorva and Ji (\citeyear{BDHJ}), Bai, D\c{e}bicki, Hashorva and  Luo (\citeyear{BDHL}). Among all known techniques for the study of the asymptotic behaviour in \eqref{question}, only the Rice method of moments may yield the behaviour of further terms. But, to the best of our knowledge, the application of the Rice method to non-stationary processes is not described in the literature. 

Interestingly enough, the knowledge of the second term in the asymptotics of~\eqref{question} can be rather useful in some statistical problems. Assume that we are given by a sample   \(X_1,...X_n\) drawn from some absolutely continuous distribution, and  we wish to estimate the density \(p(x)\) of this distribution.  More precisely, for any \(\alpha \in (0,1),\) we aim to construct \((1-\alpha)\)- confidence sets \(\C_n(x)\) for \(p\) that are \textit{honest } to a given class \(\F\) of density functions in the sense 
\begin{eqnarray}\label{CB}
\inf_{p \in \F} \P \left\{
p(x) \in \C_n(x), \; x \in \R
\right\} \geq 1- \alpha + e_n,
\end{eqnarray}
where \(e_n \to 0\) as \(n \to \infty. \) 
Very often, the set \(\C_n(x)\) is constructed using an estimate \(\hat{p}_n(x)\) of \(p(x).\) In this respect, the confidence bands can be used for showing the quality of the estimates: the narrower is the confidence band, the better is the estimate. 

%

Typically, the construction of confidence bands is based on the so-called SBR-type ("Smirnov-Bickel-Rosenblatt") limit theorems, which yield the asymptotic behaviour of the maximal deviation of the considered estimate \(\hat{p}_n(x)\) in terms of 
\begin{eqnarray*}
\DDD_{n} = \sup_{u \in \R}
\frac{\left|
\hat{p}_n(u) -  p(u)
\right|
}{
\sqrt{p(u)}
}.
\end{eqnarray*}
The SBR-type theorems state that 
\begin{eqnarray}
\label{class}
\sup_{p \in \F} \left|
\P\left\{
\DDD_n
\leq \frac{
x
}{ a_n
}
+b_n
\right\}
- 
e^{-e^{-x}}
\right| 
\to 0
\end{eqnarray}
for some deterministic sequences \(a_n\) and \(b_n\) tending to infinity as \(n \to \infty\), 
see Smirnov (\citeyear{Smirnov}), Bickel and Rosenblatt (\citeyear{Bickel}), Gin{\'e}, Koltchinskii and Sakhanenko (\citeyear{GKH}), Gin{\'e} and Nickl (\citeyear{GN10}),  Bull (\citeyear{Bull}).

All of the aforementioned  papers deal with the case when \(\hat{p}_n(x)\) is either a kernel estimate or  \textit{certain} wavelet projection estimate (e.g., based on Haar wavelets or Battle-Lemarie wavelets).  In fact, there is one rather serious technical difficulty for the proof of the SBR-type theorem for any density estimate. To explain the nature of this difficulty, let us first mention that many density estimates can be represented as 
\begin{eqnarray}
\label{pnn}
\hat{p}_n (x)  = \int_\R \mathcal{K}(x,y) d\P_n(y),\end{eqnarray}
where \(\mathcal{K}\) is a kernel (typically depending on some parameters) and  \(\P_n=n^{-1} \sum_{i=1}^n \delta_{X_i}\) is the empirical measure. The proofs of all SBR-type theorems for \eqref{pnn} are based on the idea to show that the distribution of \(\DDD_n\) for 
\(\hat{p}_n(x)\) is (in some sense) close to the distribution function 
of the supremum of the Gaussian process 
\begin{eqnarray}\label{gaussian}
\Upsilon(x)  = \int_\R \mathcal{K}(x,y) dW(y),
\end{eqnarray}
where \(W\) is a Brownian motion. For instance, for  the kernel density estimates  \(\mathcal{K}(x,y) =\mathcal{K}_h (x,y)=K((x-y)h^{-1}) h^{-1}\) with some  \(K:\R \to \R\), \(h>0\), the process \(\Upsilon(x)\) is stationary for any \(h\). For  wavelets, \(\mathcal{K}(x,y) =\mathcal{K}_j(x,y) = 2^j\sum_k \phi(2^j x-k) \phi(2^j y-k),\) with the mother wavelet \(\phi\), and the resulting process  \(\Upsilon(x)\) turns out to be non-stationary.  Nevertheless, as it is shown by Gin{\'e} and Nickl (\citeyear{GN10}, \citeyear{GN}), in the case of wavelets, this process possesses the property of \textit{cyclostationarity} in the sense that the covariance function 
\(
r(x,x+u)=\cov(\Upsilon(x), \Upsilon(x+u))
\) is periodic in \(x\) with the same period for all \(u.\) Therefore,  the further analysis of \(\Upsilon(x)\) is based on well-developed extreme value theory for the processes of this type, see  Konstant and Piterbarg (\citeyear{KP93}), H{\"u}sler, Piterbarg,  and Seleznjev  (\citeyear{HPS}).  

To the best of our knowledge, the limit behaviour of the distribution function of  \(\DDD_n\) is not known for other projection estimates, which are not related to stationary or cyclostationary Gaussian processes (see some discussion in Gin{\'e} and Nickl, \citeyear{GN10}, and in  Chernozhukov, Chetverikov and Kato, \citeyear{CCK}).  In particular, surprisingly, such results are even not known for  projection estimates based on Legendre polynomials. This fact serves as the main motivation of our research. 

The rest of the paper is organized as follows. In Section~\ref{newres}, we provide a new theoretical result (Theorem~\ref{thm2} in Section~\ref{newres1}) revealing the asymptotic behaviour of nonstationary Gaussian processes. This result yield  the asymptotic decomposition of \eqref{question}  with exponential decay of the remainder term,  and this decomposition  will be used for the statistical applications presented below. Particular attention is drawn to the process~\eqref{gaussian}  with \(\mathcal{K}(x,y) = \sum_{j=0}^J \psi_j(x) \psi_j(y),\) where \(\psi_0,\psi_1,...\) are Legendre polynomials on \([-1,1],\) see Section~\ref{exmain}.

Section~\ref{statistics} deals with the statistical applications of Theorem~\ref{thm2}.  In Section~\ref{coll}, we discuss the projection density estimates based on the idea to divide the support of the density function into \(M\) subintervals and tend \(M\) to infinity.  In Section~\ref{Gausspr}, we 
formulate an important proposition (Proposition~\ref{prop1}), which clarifies the relation of the asymptotic behaviour of the maximal deviation of this estimate to the extreme value theory for Gaussian processes. 

In Section~\ref{Sal}, we use Theorem~\ref{thm2} and Proposition~\ref{prop1} for the construction of  a sequence of distribution functions \(A_M(x)\)  (sequence of accompanying laws), which approximates the distribution function of the maximal deviation at polynomial rate, that is,
 \begin{eqnarray}\label{mmain}
\sup_{x \in \R} 
\left| 
\P \left\{
		\sqrt{\frac{n}{M}}
\DDD_{n}
		\leq x
	\right\}
	- A_{M}(x)
\right| 
	&\leq&
	\bar{c} n^{-\gamma}.
\end{eqnarray}
for some positive constants \(\bar{c}\) and \(\gamma.\) Next, we construct the confidence bands for the density \(p(x)\), which are honest at polynomial rate to some classes of densities.  Finally, in Section~\ref{exex}, we provide a numerical example. 

In this paper, we also provide the Smirnov-Bickel-Rosenblatt-type theorem for the projection density estimates (Appendix~\ref{SBR}).  It would be a worth mentioning that the property~\eqref{mmain} is much more useful than the SBR- type theorem, because  the rates of convergence in the SBR-type theorem are of logarithmic order. Let us also mention that the construction of honest confidence bands in Section~\ref{Sal} is completely based on~\eqref{mmain}. 

All proofs are collected in Section~\ref{proofs}. Appendix~\ref{A} contains some comments on Theorem~\ref{thm2}.  

\section{Asymptotic behaviour of the maximum of Gaussian processes}\label{newres}

\subsection{Main result}\label{newres1}
In this section, we consider the asymptotic behaviour of \(P_{u}(|X|,[A,B])\) for a non-stationary Gaussian process  $X(t),$ $t\in\lbrack A,B],$  with zero mean and twice a.s. differentiable trajectories. For the ease of presentation, we first show the result for the case, when the variance of \(X(t)\) has only one point of maximum on \([A,B]\). Later, in Corollary~\ref{cc} we generalise our result to the situation when the number of points of maximum is finite. 

For a set $M\subseteq
\lbrack A,B],$ which is a closure of an open set, 
we denote by  $N_{u}^{+}(M)$ the number of up-crossings of \(X(t)\) of the level \(u\) on the set \(M.\) Assume that for any \(t \in [A,B],\) the density function of $(X(t),X^{\prime}(t))$ exists, and let us denote it by 
$p_{t}(u,x)$. It holds
\begin{equation}
\E N_{u}^{+}(M)=\int_{M}\int_{0}^{\infty}xp_{t}(u,x)dxdt, \label{EN}%
\end{equation}
see Section~14.3 from~\cite{Piterbarg20}.

\begin{theorem}
\label{thm2} Assume that the process $X(t),$ $t\in\lbrack A,B]$  has zero mean and twice a.s. differentiable trajectories.  Assume also that the variance $\sigma^{2}(t)$ of  $X(t)$ attains its maximum at only one point, $t_{0}\in\lbrack A,B].$ Denote $S:=\sigma
^{2}(t_{0}).$ 
Moreover, assume non-degeneracy of the derivative, namely, assume that  if $\sigma^{\prime}%
(t_{0})=0,$ then
\begin{equation}
\P \left\{  X^{\prime}(t_{0}) = 0\right\} <1. \label{der}
\end{equation} For $\delta>0,$ denote \emph{ }
\[
\M(\delta):=\left\{  t\in\lbrack A,B]:\sigma^{2}(t)\geq\frac{S}{1+\delta
}\right\}.
\]
Then there exists some small \(\delta>0\) and some \(\chi>0\) such that   
\begin{eqnarray}\label{fmain}
\P\Bigl\{\max_{t\in\lbrack A,B]}\bigl|X(t)\bigr|\geq u\Bigr\} 
=
\pp(u) +O\Bigl(e^{-u^{2}(1+\chi)/(2S)}\Bigr),\qquad
u\rightarrow\infty,
\end{eqnarray}
where the function \(\pp:\R_+ \to \R\) is defined as follows: 
\begin{enumerate}[(i)]
\item if $t_{0}=A$ or $t_{0}=B$, and $\sigma^{\prime}(t_{0})\neq0, $ then
\begin{equation}
\pp(u)=2 \P\Bigl\{X(t_{0})\geq
u\Bigr\}; \label{pmax0}%
\end{equation}
\item if $t_{0}=A$ and $\sigma^{\prime}(t_{0})=0,  $ then
\begin{eqnarray}
\pp(u)=2\P\Bigl\{X(t_{0}%
)\geq u\Bigr\}+2\E \left[ N_{u}^{+}(\M(\delta))\right]; \label{pmax00}%
\end{eqnarray}
\item if $t_{0}\in(A,B]$ and $\sigma^{\prime}(t_{0})=0,$ then
\begin{equation}
\pp(u) =2\E \left[ N_{u}%
^{+}(\M(\delta))\right].\label{pmax}%
\end{equation}
\end{enumerate}
\end{theorem}
\begin{proof}
The proof is given in Section~\ref{main}. 
\end{proof}
\begin{remark}
The first step in the proof of  this theorem is to show that   formula~\eqref{fmain} is valid  with an appropriate function \(\pp: \R_+ \to \R\) \textit{for any \(\delta>0\)} and some \(\chi>0\) depending on \(\delta\). Namely, we show  the equality
\begin{equation*}
P_{u}(|X|,[A,B])=2P_u(X,M(\delta))+O\left(  e^{-u^{2}(1+\chi)/2}\right), \qquad u \to \infty.
\end{equation*}
 Formulas \eqref{pmax0}, \eqref{pmax00}, \eqref{pmax} are proved under the assumption that \textit{\(\delta\) is small enough}. Therefore, instead of the interval \([A,B]\), one can consider any smaller closed interval containing the point \(t_0\).
\end{remark}
\begin{remark}
It would be a worth mentioning that  the dependence between \(\chi\) and \(\delta\) may be complicated. The conditions on  $\chi$ are given in\ (\ref{chi0}, \ref{chi00}) for the case
(i), and in (\ref{chi_final}) for the cases (ii), (iii). We will see from the proof that the optimisation (maximisation) of \(\chi\) depending on \(\delta\) is possible, provided that the covariance function of \(X(t)\) is known. One example of such optimisation is given in Appendix~\ref{A}.

\end{remark}

\begin{remark} \label{rem2} 
The difference between (\ref{pmax00}) and (\ref{pmax}) for the cases   $t_{0}=A$  and $t_{0}=B$  ($\sigma^{\prime} (t_{0})=0$) can be explained by the fact that these expressions are obtained by the calculation of the number of up-crossings of the level. In fact, in the first case the point \(A\) gives significant contribution to the asymptotical behaviour, while in the second case the point \(B\) doesn't contribute to the asymptotics.  Moreover, with high probability the down-crossing after the up-crossing "will not have time" \;to happen, and the event  $\{X(B)>u\}$ will occur. If \(t_0=B,\) one can use also the equality (\ref{pmax00}) with the transformation of time  $t\Rightarrow B-t.$ After this transformation the up-crossings in the original time are transformed into the down-crossings. Therefore, the equality \eqref{pmax00} holds in the reverse time. \end{remark}
Now let us turn towards more general case.

\begin{corollary}\label{cc} 
Let $X(t),$ $t\in\lbrack A,B],$ be a Gaussian process with zero mean. Denote the  correlation function of \(X(t)\) by $\rho(s,t).$ Assume that the trajectories of $X(t)$ are twice
a.s. differentiable. Assume also that the variance of the process \(X(t)\) reaches its maximum at finite number of points, say $k$ points. Let us choose disjoint intervals $\M_{i},$ $i=1,...,k,$ each containing only one point of maximum of variance, such that for any \(i,j =1..k\), \(i \ne j,\)
\begin{eqnarray}\label{rrho}
\max_{(s,t)\in \M_{i}\times \M_{j}}\rho(s,t)<1.
\end{eqnarray}
Then there exists some $\chi>0$ such that
\begin{eqnarray}\label{kjl}
\P\Bigl\{\max_{t\in\lbrack A,B]}\bigl|X(t)\bigr|\geq u\Bigr\}=\sum_{i=1}
^{k} \pp_i(u)+O\Bigl(e^{-u^{2}
(1+\chi)/(2S)}\Bigr),
\end{eqnarray}
where \(\pp_i(u)\) are defined in Theorem~\ref{thm2} applied to  \(\P\Bigl\{\max_{t\in \M_{i}}\bigl|X(t)\bigr|\geq u\Bigr\}.\)
\end{corollary}
\begin{proof} The proof is given in Section~\ref{main}. 
\end{proof} 
\begin{remark}\label{corcor}
Under the assumptions of Corollary~\ref{cc}, we get 
\begin{eqnarray}
\P \Bigl\{
\max_{t \in [A,B]} |X(t)| > u 
\Bigr\} 
&=&
\frac{
\ccc_0
}{ u
}
e^{-u^2/(2 S)} 
\bigl( 
1+ o(1)
\bigr) 
\label{Gauss} 
\end{eqnarray}
as \(u \to \infty,\) where \(\ccc_0>0\). 
This form is rather classical and can be obtained also using the Pickand or Berman methods. For instance, Theorem~D.3 and Corollary~8.3 from \cite{Piterbarg} yield the same result. Nevertheless, for the statistical applications presented below the rates of convergence in \eqref{Gauss} should be clarified, and therefore we need more precise form~\eqref{kjl}.
\end{remark}


\subsection{Example}\label{exmain}
As we will see below,  the study  of the maximal deviation of projection density estimates is closely related to the analysis of the Gaussian process
\begin{eqnarray}\label{Ups}
\Upsilon(t) = \sum_{j=0}^J \psi_j (t) Z_j,
\end{eqnarray}
where \(Z_0, Z_1,...\) are i.i.d. standard normal random variables, and \(\psi_0, \psi_1,...\) is a basis in the space \(L^2([A,B]).\) The variance of the process \(\Upsilon(t)\) is equal to 
\begin{eqnarray}\label{vvar}
\sigma^2(t) = \Var(\Upsilon(t))= \sum_{j=0}^J \psi^2_j (t). 
\end{eqnarray}
Let us consider more precisely the case of  Legendre polynomials,  which are  defined on \([-1,1]\) as 
\begin{eqnarray*}
\psi_j(x) =  \sqrt{(2 j +1) / 2 } \cdot P_j (x), \qquad j=0,1,2,...,
\end{eqnarray*}
where
\begin{eqnarray*}
		P_{j}(x)= \frac{1}{j! 2^{j}}\left[ \left( x^{2}-1\right) ^{j}\right] ^{(j)}, \quad |x|\leq 1.
\end{eqnarray*}
Maximum of each function \(\psi_j(t)\), \(j=0,1,..\) (as well as maximum of the variance \(\sigma^2(t)\)) is attained at two points, \(1\) and \(-1,\)
and 
\begin{eqnarray*}
S= \sum_{j=0}^J \psi_j^2 (1) = \frac{(J+1)^2}{2
}.
\end{eqnarray*}
Now let us apply Corollary~\ref{cc} with  \(\M_1 = [-1,-2/3]\) and \(\M_2=[2/3,1]\). Note that the condition~\eqref{rrho} is fulfilled,  
\begin{eqnarray*}
\rho(s,t)=\frac{\sum_{j=0}^J \psi_j(s) \psi_j(t)}{\sqrt{\sum_{j=0}^J \psi_j^2(s)}\cdot \sqrt{\sum_{j=0}^J \psi_j^2(t)}} <1, \qquad \forall t \ne s.
\end{eqnarray*}
The process \(X_t\), considered for \(t \in \M_1\) and \(t \in \M_2,\) corresponds to the case~(i) of Theorem~\ref{thm2}, and 
\begin{eqnarray*}
\pp_1(u) = \pp_2(u) = 2 \P \Bigl\{
\Upsilon(1) \geq u
\Bigr\}=
2 \Bigl(
 1 - \Phi\bigl(
 	u/\sqrt{S}
 \bigr)
\Bigr).
\end{eqnarray*}
Therefore there exists some \(\chi>0\) such that 
\begin{eqnarray*}
\P\Bigl\{\max_{t\in\lbrack -1,1]}\bigl|\Upsilon(t)\bigr|\geq u\Bigr\}
&=&
4 \Bigl(
 1 - \Phi\bigl(
 	u/\sqrt{S}
 \bigr)
\Bigr)+O\Bigl(e^{-u^{2}
(1+\chi)/(2S)}\Bigr).
\end{eqnarray*}
Note that due to the decomposition 
\begin{eqnarray*}
1- \Phi(u) = \frac{1}{\sqrt{2\pi} u} e^{-u^2/2} ( 1- u^{-2} + o(u^{-2})), 
\end{eqnarray*}
we get  \(\mathfrak{c}_0=4\sqrt{S/(2\pi)} = 2(J+1)/\sqrt{\pi}\) in \eqref{Gauss}.

\section{Asymptotic behaviour of the maximal deviation of projection density estimates}\label{statistics}
  Throughout the section we  assume that the function \(p\) belongs to the space  \(L^2(I)\) with \(I=[A,B]\). Let \(\Psi := \bigl \{ \psi_0, \psi_1,\psi_2,...\bigr\}\) be an orthonormal basis of this space.

\subsection{Projection estimates} 
\label{coll} 
  
Let us divide the interval on \(M\) subintervals of length \(\delta=(B-A)/M,\)  and on each subinterval \(I_{m}=[a_m, b_m]:= [A+\delta(m-1), A+\delta m], m=1..M\),  we reproduce the basis: 
\begin{eqnarray}
\label{psij}
\psi_j^{(m)} (x) = \sqrt{M} \cdot
\psi_j \Bigl(
	M(x-a_m) + A
\Bigr), \quad  m=1..M,\;\;  j=0,1,... .
\end{eqnarray}
Since \(p \in L^2([A,B]),\) one can project \(p\) onto the constructed basis and get for any \(M,\)
\begin{eqnarray*}
p(x) = \sum_{m=1}^M \sum_{j=0}^\infty \left[ 
\int \psim(u) p(u) du 
\right]
 \psim(x).
\end{eqnarray*}
This formula suggests the following definition. Given the sample \(X_1,..., X_n\), the  projection estimate of the density function \(p\) is defined as 
\begin{eqnarray}\label{projest}
\hat{p}_n(x) &=& \sum_{m=1}^M 
\sum_{j=0}^J 
\left[ 
\int \psim(u)d\P_n(u) 
\right]
 \psim(x),
\end{eqnarray}
where  \(J \in \N\). 

It would be a natural question why  we do not consider more simple construction, namely the projection estimate on the original basis  \(\psi_0, ..., \psi_J,\) and then tend \(J\) to infinity. The answer lies in purely technical area: it turns out that the main step in the construction of confidence intervals - studying the asymptotics of the maximum of corresponding Gaussian process - is more complicated. We will discuss this issue later, see Remark~\ref{rem3}.  

So, throughout this paper \(J\) is fixed and \(M\) tends to infinity as \(n\) grows. \newline

For the theoretical study, we need the following assumptions on the basis~\(\Psi\).
\begin{enumerate}
\item[(A1)] For any \(j=0..J, \) the function \(\psi_j\) is uniformly H\"older continuous with some exponent \(\alpha \in (0,1]\), that is, the \(\alpha-\)H\"older coefficient of \(\psi_j\)
\begin{eqnarray*}
\bigl| 
\psi_j
\bigr|_{\alpha} := \sup_{x \ne y, \; x,y \in [A,B]} \frac{|\psi_j(x) - \psi_j(y)|}{|x-y|^{\alpha}}
\end{eqnarray*}
is finite.

\item[(A2)] The maximum of the sum \(\sum_{j=0}^J \psi_j^2(x)\)  is attained at a finite number of points, which we denote below by \(x_\circ^{(1)}, ..., x_\circ^{(k)}.\) 

%
\end{enumerate}
\begin{example}\label{legendre}
Let us consider once more the example of Legendre polynomials, defined above in Section~\ref{exmain}. 
Condition (A1) holds with \(\alpha=1\) since the functions \(P_j(x)\) are continuously differentiable.  Moreover, absolute values of the polynomials \(P_j(x)\) attain its maximum, which is equal to \(1,\) at the points 1 and -1, and therefore the assumption (A2) is also fulfilled. 
\end{example}

\subsection{Relation to the extreme value theory for Gaussian processes}\label{Gausspr}
Let us first study the  "random part" of \(\DDD_n\) defined by 
\begin{eqnarray}
\label{RR}
	\RR_{n} &:=&		\sup_{u \in [A,B]} 
		\frac{				\left|
		\hp_{n} (u)  - \E[\hp_{n}(u)]  \right|
			}
			{		\sqrt{p(u)}
			},
\end{eqnarray}
and find a Gaussian process having the asymptotic behaviour of maximum closely related to the behaviour of \(\RR_n\). As we will show below in Proposition~\ref{prop1},  the Gaussian process is 
\begin{eqnarray}\label{zeta}
\p(x)=\int_{I} \Bigl( \sum_{j=0}^J 
\psi_j(x) \psi_j(u)
\Bigr) dW(u)= \sum_{j=0}^J \psi_j(x) \xi_j
 \qquad x \in [A,B],
\end{eqnarray}
where \(W\) is a Brownian motion, and \(\xi_0, ..., \xi_J\) are i.i.d. standard normal random variables. 

It would be a worth mentioning that all constants in Proposition~\ref{prop1} are uniform over the class \(\mathcal{P}_{q,H, \beta}\) of densities, which is defined for any \(q>0, \\H>0, \beta \in (0,1]\) as 
\begin{multline}\label{class}
\mathcal{P}_{q,H,\beta} := \Bigl\{ p - \mbox{ p.d.f. }, \quad p \in L^2([A,B]),\quad \inf_{x \in [A,B]} p(x) \geq q, 
\quad 
\bigl| 
p
\bigr|_{\beta} \leq H 
\Bigr\},
\end{multline}
where \(  \bigl|  p \bigr|_{\beta} \)
is the H{\"o}lder coefficient of
 the function \(p\). 
 \begin{proposition}\label{prop1} 
Consider the projection estimate~\eqref{projest} on \([A,B]\) constructed on the basis  \(\Psi\)  satisfying the condition (A1). 
Then there exists a positive constant \(\kappa\)  such that for any \(p \in \PP_{q,H,\beta}\) and any \(u \in \R_+\) it holds 
\begin{eqnarray}
\label{Fst1}
		\P \Bigl\{\sqrt{\frac{n}{M}}\RR_{n} \leq u\Bigr\}
		& \leq & 
		\left[
			\P \Bigl\{ 
				\zeta \leq
					 u + \gamma_{n,M}
			\Bigr\} 
		\right]^{M} +  \C_{1} n^{-\kappa},\\
		\label{Fst2}
		\P \Bigl\{\sqrt{\frac{n}{M}}\RR_{n} \leq u\Bigr\}
		& \geq & 
		\left[
			\P \Bigl\{ 
				\zeta \leq
					 u - \gamma_{n,M}
			\Bigr\} 
		\right]^{M}  -  \C_{1} n^{-\kappa},			
		\end{eqnarray}
 where \(\zeta := \sup_{x \in \R} |\Upsilon(x)|\), \[\gamma_{n,M} = \C_2 \frac{\log(n)}{\sqrt{n/M}} + \C_{3} \frac{\sqrt{\log(n)}}{\sqrt{M}},\]
 and \(\C_1, \C_2, \C_3 >0\) depend on \(q,H,\beta\).
\end{proposition}
 \begin{proof}
 The proof is given in Section~\ref{App1}. 
\end{proof}
\begin{remark}\label{rem1}
The most important case arises when \(n/M \to \infty\) and \(\gamma_{n,M}  \to 0\)  as \(n\) grows to infinity. For instance,  these conditions are fulfilled when \(M=M_n=\lfloor n^\lambda \rfloor\) with \(
\lambda \in (0,1)\). Under this choice, \(\gamma_{n,M}\asymp \C_2  n^{-(1-\lambda)/2} \log(n)\) if \(\lambda \in [1/2,1)\) and \(\gamma_{n,M} \asymp \C_3 n^{-\lambda/2} (\log(n))^{1/2}\) if \(\lambda \in (0,1/2)\).
\end{remark}
\begin{remark}\label{rem3} 
As it was mentioned at the end of Section~\ref{coll}, one can consider a simplified version of the estimate \(\hat{p}_n\), which is constructed without the splitting of the interval \([A,B]\) into \(M\) small subintervals. In this case, applying similar techniques, we can also show that  \(\RR_n\) is close  to \(\Upsilon(x)\), but the further study of \(\RR_n\) would be essentially different. 
We believe that this case merits separate publication.  
\end{remark}
It would be a worth mentioning that the process \(\Upsilon(x)\) is not cyclostationary. Therefore the asymptotic behaviour of the supremum of \(\Upsilon(x)\) cannot be studied following the same ideas as in previous papers on this topic, and new results on the behaviour of maximum of \(\Upsilon\) (namely, Theorem~\ref{thm2}) are needed. In the next two subsections, we show some results obtained via the combination of Theorem~\ref{thm2} and Proposition~\ref{prop1}.


\subsection{Sequence of accompanying laws} \label{Sal}
In the next theorem,  we will present a sequence of accompanying laws, which approximates the maximum deviation distribution at polynomial rate. The proof of this statement is essentially based on  Theorem~\ref{thm2}. 
\begin{theorem} \label{thm35}  
Assume that \(p\in \mathcal{P}_{q,H,\beta}\) with some  \(q,H>0, \; \beta \in (0,1]\), and assume that the basis functions \(\psi_j(x),  j=0,1,2,...\) satisfy the assumptions (A1) and (A2). Denote by \(\pp_i(u)\) the functions introduced in Corollary~\ref{cc} for  the process \(\Upsilon\) defined by \eqref{zeta}. Denote the sequence of distribution functions
\begin{equation}
\label{amy}
A_{M}(x):=\begin{cases}
\exp\Bigl\{
-M \sum_{i=1}^k \pp_i\bigl( 
x
\bigr) 
\Bigr\}, 
&\mbox{if }x\geq c_M, \\ 
0,  &\mbox{if }x< c_M,%
\end{cases}   
\end{equation}
where \(c_M=(2S \log{M})^{1/2} -S.\)
Assume that for the sequence \(M=M_n\) there exists some \(v>0\) such that 
\begin{eqnarray}\label{coord} 
V_{n,M}(v) := M^{v} \gamma_{n,M}
 \to 0, \qquad \mbox{as} \;\; n,M \to \infty.
\end{eqnarray}
\begin{enumerate}[(i)]
\item 
Then there exist some positive constants \(\bar{c}_1, \bar{c}_2, \bar{c}_3, \bar{c}_4, \theta, \kappa \) such that   for sufficiently large \(n,M\) and for any \(x\in \R\),
\begin{multline}\label{amm}
\sup_{x \in \R} 
\left| 
\P \left\{
		\sqrt{\frac{n}{M}}
\RR_{n}
		\leq x
	\right\}
	- A_{M}(x)
\right| 
	\\ \leq
	\bar{c}_1\; M^{-\theta}+
	\bar{c}_2\; n^{-\kappa}
	+
	 \bar{c}_3 V_{n,M}(v)
	 +\bar{c}_4 n^{1/2} M^{-\beta-1/2}.
\end{multline}
\item 
In particular, if \(M=\lfloor n^\lambda \rfloor\) with \(\lambda \in ((2 \beta+1)^{-1},1)\), the assumption~\eqref{coord} is fulfilled, and moreover, 
\begin{eqnarray*}
\sup_{x \in \R} 
\left| 
\P \left\{
		\sqrt{\frac{n}{M}}
\DDD_{n}
		\leq x
	\right\}
	- A_{M}(x)
\right| 
	&\leq&
	\bar{c} n^{-\gamma}.
\end{eqnarray*}
for some positive constants \(\bar{c}\) and \(\gamma.\)
\end{enumerate}
\end{theorem}
\begin{proof}
The proof is given in Section~\ref{app4}. 
\end{proof}

Theorem~\ref{thm35} states that under certain assumptions on the basis \(\Psi,\) and the choice \(M=\lfloor n^\lambda \rfloor\) with \(\lambda \in ((2 \beta+1)^{-1},1)\), we have 
\begin{eqnarray}
\label{Pret}
\P \left\{
\sqrt{\frac{n}{M}}
\DDD_{n}
		\leq 
u_M(x)	\right\}
= 
 A_{M}(x) + \eps_{n}(x),
\end{eqnarray}
where \(|\eps_{n}(x)|\leq \bar{c}n^{-\gamma} \) and \(\bar{c}, \gamma\) do not depend on \(x\) and \(p.\)

Let us show how~\eqref{Pret} can be used  for the construction of honest confidence bands. Let us fix some confidence level \(\alpha \in (0,1)\) and 
denote the \((1-\alpha)\)- quantile of the  distribution function \(A_{M}(\cdot)\)  by \(q_{\alpha,M}\). Then
\begin{eqnarray*}
\P \Bigl\{
			\frac{
				\left|
		\hat{p}_{n} (x)  - p (x)  
				\right|
			}
			{	
				\sqrt{p(x)}
			}
			\leq k_{\alpha,M}, \;\; \forall x \in [A,B]
			\Bigr\}= 1 - \alpha +e_{n,M},
\end{eqnarray*}
with \(k_{\alpha,M} :=\sqrt{M/n}\cdot q_{\alpha, M}\) 
and \(e_{n,M} = \eps_{n}(q_{\alpha,M})\). Note that \(e_{n,M}\) converges to zero at polynomial level in \(n\).
Solving the corresponding quadratic inequality with respect to \(\sqrt{p(x)}\), we obtain that 
\begin{multline*}
\C_n(x):= 
\biggl(
\hat{p}_n(x) +
( k^2_{\alpha,M}/2) -
\bigl[
\hat{p}_n(x)k^2_{\alpha,M}
+(k_{\alpha,M}^4/4)
\bigr]^{1/2}, \\
\hat{p}_n(x) +
( k^2_{\alpha,M}/2) +
\bigl[
\hat{p}_n(x)k^2_{\alpha,M}
+(k_{\alpha,M}^4/4)
\bigr]^{1/2}\biggr)
\end{multline*}
is a \((1-\alpha)-\)confidence band for \(p(x)\). This confidence set is honest to the class \(\mathcal{P}_{q,H}\) for any positive \(q,H\), \(\beta \in (0,1]\), at polynomial rate with respect to both \(n\) and \(M.\)

\subsection{Numerical example} \label{exex}
Let us consider the density 
\begin{eqnarray}\label{BS}
 p(x) = \frac{1}{2}\phi_{(0,1)}(x) + \frac{1}{10} \sum_{j=0}
 ^4 \phi_{((j/2)-1, 1/100)}(x),
\end{eqnarray}
where \(\phi_{(\mu,\sigma^2)}\) stands for the normal distribution with mean \(\mu\) and variance \(\sigma^2.\) This density function is known as "the claw" \cite{MW} or "Bart Simpson density" \cite{W2006} - the origins of these names are clear from the plot of \(p\), see Figure~\ref{Fig6}. We simulate a sample \(X_1,...,X_n\)  from this density, and aim to construct a confidence band for \(p.\)

\begin{figure*}
\begin{center}
\includegraphics[width=1\linewidth ]{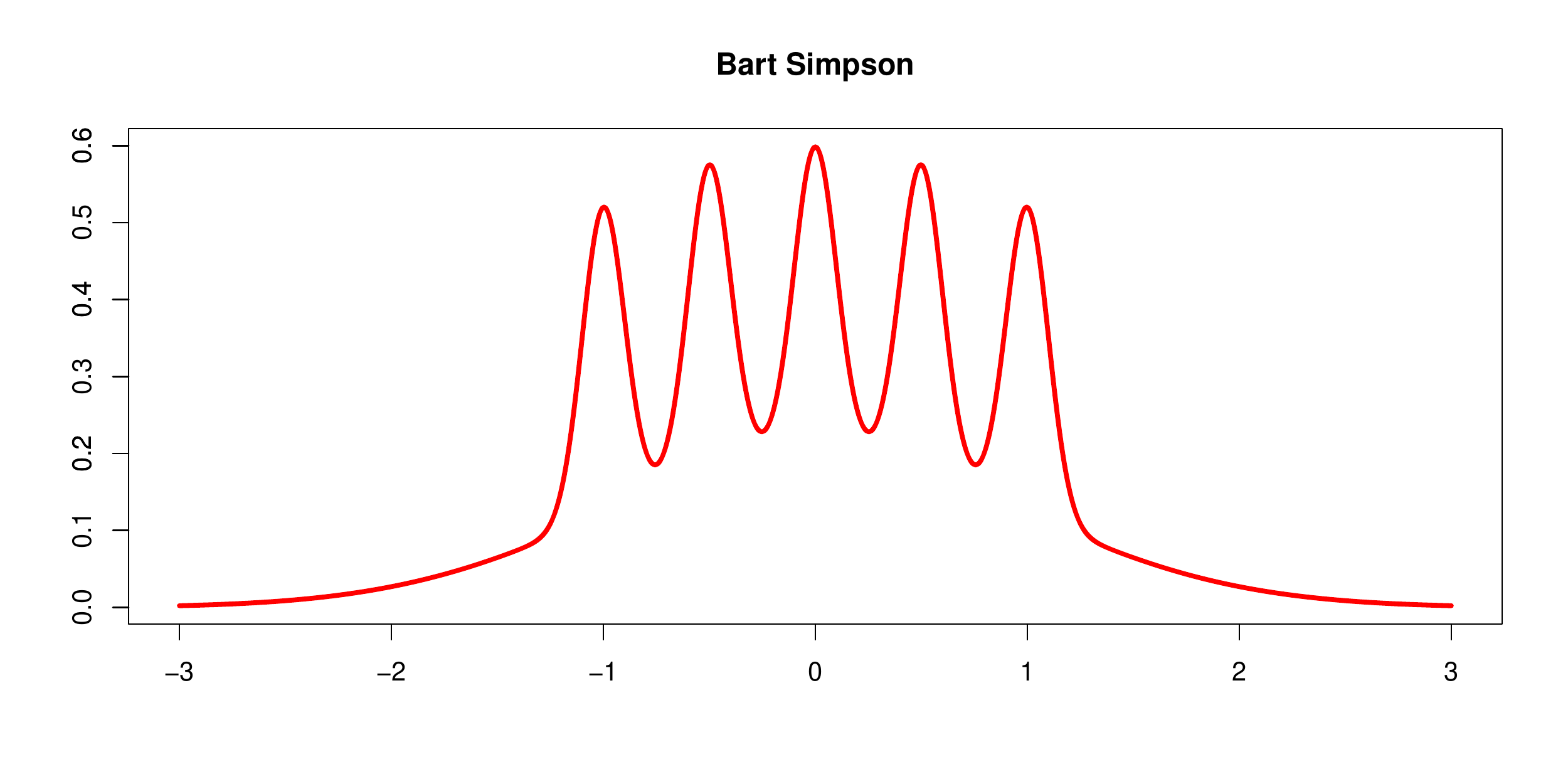}\caption{ Plot of the density function~\eqref{BS}. } \label{Fig6}
\end{center}
\end{figure*}

Let us consider the  orthonormal basis on \(L^2([-3,3]),\) defined as 
\(\tilde{\psi}_j(x) = 3^{-1/2} \psi_j(x/3),\) where \(\psi_j, j=0,1,2,...\) are the  Legendre polynomials on \([-1,1],\) see Section~\ref{exmain}. 
Following the scheme described in Section~\ref{coll}, we divide the interval \([-3,3]\) into \(M\) subintervals \(I_m:=[a_m, b_m]\) of the same length \(\delta=6/M\), and 
define the projection estimate 
\begin{eqnarray*}
\hat{p}_n(x) &=& 
\frac{1}{n}
\sum_{j=0}^J 
\Bigl[
\sum_{i: X_i \in I_m}
\tilde{\psi}_j^{(m)}(X_i)
\Bigr]
\tilde{\psi}_j^{(m)}(x), \qquad x \in I_m,
\end{eqnarray*}
where for \(x \in I_m,\)
\begin{eqnarray*}
\tilde\psi_j^{(m)}(x) = 
\sqrt{M/3} \cdot
\psi_j \Bigl(
	\frac{M(x-a_m)}{3} -1
\Bigr)
, \quad  m=1..M,\;\;  j=0,1,...,J.
\end{eqnarray*}
The corresponding Gaussian process  is defined by~\eqref{zeta}  as 
\begin{eqnarray*}
\Upsilon(x) = 3^{-1/2} \sum_{j=0}^J \psi_j (x/3) Z_j,
 \qquad x \in [-3,3],
\end{eqnarray*}
where \(Z_0, Z_1,...\) are i.i.d. standard normal random variables. 
Following the results of Section~\ref{exmain}, we get  that  the variance of \(\Upsilon(x)\) attains its maximum, which is equal to \(S=(J+1)^2/6,\) in 2 points, namely in \(t=-3\) and \(t=3\), and 
\begin{eqnarray*}
\pp_1(u) = \pp_2(u) = 2 \P \Bigl\{
\Upsilon(1) \geq \sqrt{3} u
\Bigr\}=
2 \Bigl(
 1 - \Phi\bigl(
\sqrt{6} u /(J+1)
 \bigr)
\Bigr).
\end{eqnarray*}
The sequence of accompanying laws defined in~Theorem~\ref{thm35}  is equal to 
\begin{equation}
\label{amy}
A_{M}(x):=\begin{cases}
\exp\Bigl\{
-4M \Bigl(
 1 - \Phi\bigl(
 	\sqrt{6} x / (J+1)
 \bigr)
\Bigr)
\Bigr\}, 
&\mbox{if }x\geq c_{M}, \\ 
0,  &\mbox{if }x<c_{M},%
\end{cases}   
\end{equation}
where \[c_M = \frac{ (J+1)}{\sqrt{3}}  \sqrt{\log(M)} - \frac{(J+1)^2}{6}.\]
According to Theorem~\ref{thm35}, \(A_M(x)\) is close to the distribution function of \begin{eqnarray*}
\DDD_{n} = \sup_{u \in \R} 
\frac{|
\hat{p}_n(u) -  p(u)|
}{
\sqrt{p(u)}
}.
\end{eqnarray*}
The closeness of these functions is illustrated by Figure~\ref{FF}. In fact, as \(n\) grows, the distribution function of \(\sqrt{n/M_n} \DDD_n\) converges to \(A_M(x)\)
\begin{figure*}
\begin{center}
\includegraphics[width=1\linewidth ]{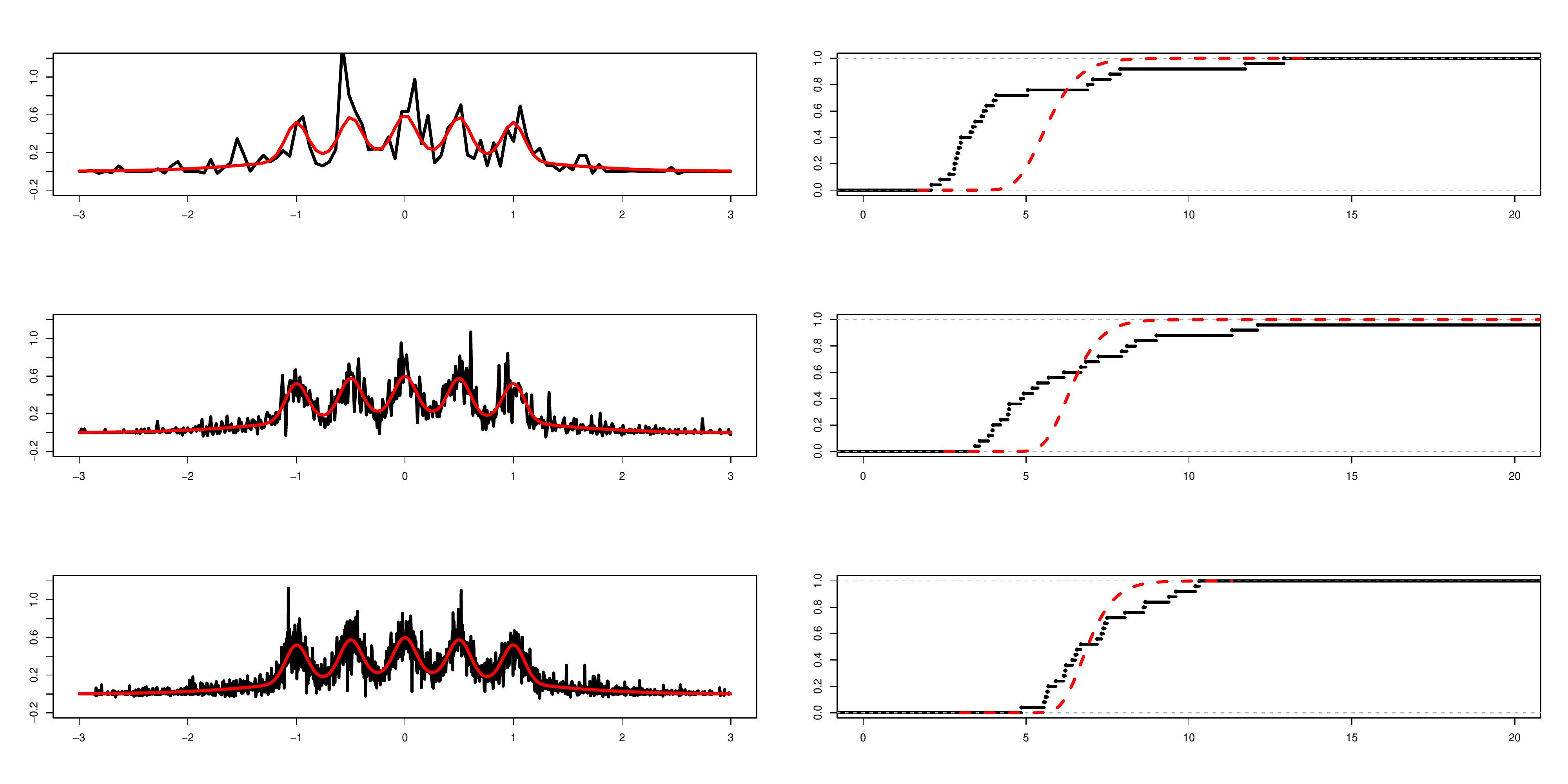}\caption{ 
First raw:  projection density estimates (black solid lines) in comparison with the true densities (red lines) based on \(n =500, 3000, 10000\) simulations. In this example, we take \(M = \lfloor n^{2/3} \rfloor.\) Second raw: empirical distribution functions of \(\sqrt{n/M_n} \cdot\DDD_n\) (black solid lines) based on 25 simulation runs in comparison with the distribution function \(A_M(x)\) (red dashed curves). Clearly, when \(n\) grows, the difference between these distribution functions decays.} \label{FF}
\end{center}
\end{figure*}


\section{Proofs} \label{proofs}
\subsection{Proof of Theorem~\ref{thm2} and Corollary~\ref{cc}}\label{main}
The proof of Theorem~\ref{thm2} consists of 3 steps. 

\textbf{Step 1. } 
Let  us assume for brevity that $S=1.$ Due to Theorem~8.1 from \cite{Piterbarg},
\begin{eqnarray}
\label{ext}
P_{u}(|X|,[A,B]\setminus M(\delta))\leq C(B-A)e^{-u^{2}(1+\delta)/2},
\end{eqnarray}
where \(P_u\) is defined by~\eqref{question}, and the constant $C$ depends on the maximum of the variance of $X'_t.$ Therefore, 
\begin{equation}
P_{u}(|X|,[A,B])=P_{u}(|X|,M(\delta))+O\left(  e^{-u^{2}(1+\delta)/2}\right)
,\ \label{opt1}%
\end{equation}
as $u\rightarrow\infty.$
 Next, since the process \(X_t\) is centred, 
\begin{multline*}
P_u(|X|,M(\delta))     =2P_u(X,M(\delta))
-\P\bigl\{\max_{t\in M(\delta)}X(t)    \geq u,\max_{t\in M(\delta)}(-X(t))\geq u\bigr\}.
\end{multline*}
The last probability in the r.h.s. can be majorized by 
\[
P_{2u}\left(  X(t)-X(s),M(\delta)\times M(\delta)\right)  ,
\]
which can be  bounded due to Theorem~8.1 from  \cite{Piterbarg} by 
\begin{align}
&  C(B-A)^{2}ue^{-4u^{2}/(2S_{1})},\text{ with }\label{opt2}\\
&  \ S_{1}=\max_{M(\delta)\times M(\delta)}(\sigma^{2}(s)+\sigma
^{2}(t)-2r(s,t))\leq2-2\min_{M(\delta)\times M(\delta)}r(s,t),\nonumber
\end{align}
where \(r(s,t)\) is the covariance function of the process \(X(t).\) Combining the formulas from this step, we arrive at 
\begin{equation}
P_{u}(|X|,[A,B])=2P_u(X,M(\delta))+O\left(  e^{-u^{2}(1+\chi)/2}\right)
,\ \label{optopt}%
\end{equation}
where
\begin{equation}
\chi<\min\left\{  \delta,\frac{1+\min_{M(\delta)\times M(\delta)}%
r(s,t)}{1-\min_{M(\delta)\times M(\delta)}r(s,t)}\right\} := \chi_1(\delta). \label{chi0}%
\end{equation}
\textbf{Step 2.}
Let us now concentrate on the case 
(i). In what follows, we assume that \(\delta>0\) is small enough and the set \(M(\delta)\) is an interval denoted below by \([a,b]\).  Denote \[ r_{kl}(s,t):=\frac{\partial^{k+l}r(s,t)}{\partial^{k}s\partial^{l}t}, \] e.g., $r_{10} (s,t)=\partial r(s,t)/\partial s.\ $ 

Let us assume that $t_{0}=A$ and $\sigma^{\prime}(t_{0})\ne0.$   In this case,
$\sigma^{\prime}(t_{0})<0.$ The proof for another case 
$t_{0}=B$ (with $\sigma^{\prime}(t_{0})>0$) follows from the time inversion argument. As we have already mentioned in Remark~\ref{rem2}, the up-crossings are transformed into the  down-crossings. The covariance between $X(t_{0})$ and $X^{\prime}(t_{0})$
is negative,
\[
r_{10}(t_{0},t_{0})=\frac{1}{2}\left.  \frac{d\sigma^{2}(t)}{dt}\right\vert
_{t=t_{0}}<0,
\]
and therefore we can assume that $\delta$ is such that \begin{equation}\label{r10}
\max_{t\in M(\delta)}r_{10}(t,t)<0. 
\end{equation}
We have 
\begin{equation}
P_u(X,M(\delta))=P(X(A)\geq u)+P(X(A)<u,N_{u}^{+}%
(M(\delta))\geq1).\nonumber
\end{equation}
Trivially,
\[
P(X(A)<u,N_{u}^{+}(M(\delta))\geq1)\leq EN_{u}^{+}(M(\delta)),
\]
where the expectation in the right-hand side can be calculated via  (\ref{EN}). Let us represent the density 
 $p_{t}(u,x)$
via the conditional density $X^{\prime
}(t)$ under the condition
$X(t)=u$:
\begin{equation} p_{t}(u,x)=\frac{1}{\sqrt{2\pi}\sigma(t)}e^{-\frac{u^{2}}{2\sigma^{2}(t)}}p_{X^{\prime }(t)|X(t)=u}(x). \label{dens} \end{equation}
The mean and the variance of conditional distribution are equal to 
\begin{equation}
m(t):=u\frac{r_{10}(t,t)}{r(t,t)},\qquad d^{2}(t):=r_{11}(t,t)-\frac{r_{10}%
^{2}(t,t)}{r(t,t)}. \label{dens1}%
\end{equation}
Next, by changing the variables 
$y=(x-m(t))/d(t)$ we get 
\[
\int_{0}^{\infty}xp_{X^{\prime}(t)|X(t)=u}(x)dx=\frac{1}{\sqrt{2\pi}}%
\int_{-m(t)/d(t)}^{\infty}(d(t)y+m(t))e^{-y^{2}/2}dy,
\]
and, since due to~\eqref{r10} we have
$m(t)<0$ for any \(t \in \M(\delta)\), it holds

\[
EN_{u}^{+}(M(\delta))\leq\frac{1}{2\pi}\int_{M(\delta)}\frac{d(t)}{\sigma
(t)}\exp\left(  -\frac{u^{2}}{2}\left(  \frac{1}{\sigma^{2}(t)}+\frac
{m^{2}(t)}{u^{2}d^{2}(t)}\right)  \right)  dt.
\]
Due to (\ref{dens1}), $d^{2}(t)\leq r_{11}(t,t),$ and moreover $r(t,t)=\sigma^{2}(t)\leq1.$ Therefore, 
\[
\frac{1}{\sigma^{2}(t)}+\frac{m^{2}(t)}{u^{2}d^{2}(t)}\geq1+\frac{r_{10}%
^{2}(t,t)}{r_{11}(t,t)},
\]
and
\[
EN_{u}^{+}(M(\delta))\leq Ce^{-u^{2}(1+\chi_{2})/2}
\]
with \[
\chi_{2}=\chi_{2}(\delta):=\min_{t\in M(\delta)}\frac{r_{10}^{2}%
(t,t)}{r_{11}(t,t)}>0.
\]
So we get one more condition on the constant \(\chi\) for the case (i):
\begin{equation}
\chi<\chi_{2}(\delta). \label{chi00}%
\end{equation}
\textbf{Step 3.} Now let us turn towards to the proof of  (ii) and (iii). Let us assume that the number $\delta$ is small enough to guarantee that the set $M(\delta)$ satisfies the conditions 
\[
\max_{t\in M(\delta)}\frac{d^{2}}{dt^{2}}\sigma^{2}(t)\leq0\ \text{and\ }%
\min_{t\in M(\delta)}r_{11}(t,t)>0.
\]
Such $\delta$ exists due to \eqref{der} and \eqref{dens1}. We have 
\begin{multline}
P_{u}(X,[a,b])  =P(X(a)\geq u)+P(X(a)<u,N_{u}^{+}([a,b])=1)\\
+P(X(a)   <u,N_{u}^{+}([a,b])\geq2)=:P(X(a)\geq u)+p_{1}+p_{2}.
\label{est1}%
\end{multline}
There exists at least one down-crossing  of the level \(u\) between two up-crossings of the level \(u\), and therefore 
\begin{align}
p_{2}  &  \leq P(N_{u}^{+}([a,b])\geq2)\nonumber\\
&  \leq P(\exists t_{1},t_{2}\in\lbrack a,b]:t_{2}>t_{1},X(t_{i})=u,X^{\prime
}(t_{i})\geq0,i=1,2)\nonumber\\
&  =P(\exists\tau\in(a,b):X(\tau)=u,X^{\prime}(\tau)<0,\max_{[\tau
,b]}X(s)>u). \label{p_2}%
\end{align}
Let us denote the number of such points  $\tau$ by
$N_{u}^{-+}([a,b]).$ So, we have proved that \(p_2 \leq \E N_{u}^{-+}([a,b]).\)

Let us show that   
\begin{eqnarray}
\label{p1} p_1 = \E N_{u}^{+}([a,b]) + O( \E N_{u}^{-+}([a,b])).\end{eqnarray} In fact, 
\begin{multline*}
EN_{u}^{+}([a,b])  = P(X(a)<u,N_{u}^{+}([a,b])=1)\\
+ P(X(a)  > u,N_{u}^{+}([a,b])=1)+\sum_{k=2}^{\infty}kP(N_{u}%
^{+}([a,b])=k),
\end{multline*}
and therefore 
\begin{multline}\label{estp_1}
p_{1}= EN_{u}^{+}([a,b])-P(X(a)\geq u,N_{u}^{+}([a,b])=1)\\
  -\sum_{k=2}^{\infty}kP(N_{u}^{+}([a,b])=k).
\end{multline}
We have 
\begin{eqnarray*}
P(N_{u}^{+}([a,b])=1) 
&\leq& \E [N_{u}^{-+}([a,b])],\\ 
\sum_{k=2}^{\infty}kP(N_{u}^{+}([a,b])=k)
&=& 
\sum_{k=2}^{\infty}kP(N^{-+}([a,b])=k-1)\\
&\leq& 
2\sum_{k=2}^{\infty} (k-1)  P(N^{-+}([a,b])=k-1)\\
 &\leq&
2  \E [N_{u}^{-+}([a,b])],
\end{eqnarray*}
and therefore we get~\eqref{p1}.

The method described in  the proof of Theorem~E.1 from~\cite{Piterbarg}  allows to calculate the mean value of the points $\tau$, which are the down-crossings of the level $u$ and satisfy
\[
\max_{s\in\lbrack\tau,b]}X(s)\geq u.
\]
We have
\begin{multline}
  EN_{u}^{-+}([a,b])\label{est3n}\\
  =\int_{a}^{b}dt\int_{-\infty}^{0}|x|dxp_{t}(u,x)P(\max_{s\in\lbrack
t,b]}X(s)>u|X(t)=u,X^{\prime}(t)=x),
\end{multline}
where the density $p_{t}(u,x)$ is given by (\ref{dens}). The formulas of this type are obtained by using the discretisation of time  $\{k2^{-n}\in M(\delta ),k,n\in\mathbb{Z}\}$ and taking the limit as   $n\rightarrow\infty,$ see \cite{Piterbarg}. 


Let us fix $\alpha>0$ (the exact value will be clarified later), and then decompose the integral in (\ref{est3n}) into two: over the set  $[a,b]\times (-\infty,-\alpha u]$ and over its compliment,  $[a,b]\times(-\alpha u,0].$ In the first integral we bound the probability by 1, and the further analysis consists only in the analysis of the two-dimensional Gaussian density by using (\ref{dens}, \ref{dens1}). For the second summand, we firstly estimate the sum of coefficients corresponding to  $(-u^{2})$ under exponent in the two-dimensional density and at the point of maximum of conditional variance of conditional process. Then we use the Laplace method. 

So,  the integral over the set $[a,b]\times (-\infty,-\alpha u]$ is upper bounded by 
\begin{equation}
\frac{1}{\sqrt{2\pi} \sigma(t)}\int_{a}^{b}e^{-\frac{u^{2}%
}{2\sigma^{2}(t)}}dt\int_{-\infty}^{-\alpha u}|x|dxp_{X^{\prime}%
(t)|X(t)=u}(x).\label{chi2}%
\end{equation}
For the conditional mean (\ref{dens1})
we have 
\begin{equation}
|m(t)|\leq u\max_{t\in M(\delta)}\frac{|r_{10}(t,t)|}{r(t,t)}=:c(\delta
)u.\label{c(de)}%
\end{equation}
Next, for 
$\alpha>c(\delta)$ we obtain 
\begin{align}
\int_{-\infty}^{-\alpha u}|x|dxp_{X^{\prime}(t)|X(t)=u}(x) &  \leq\frac
{1}{\sqrt{2\pi}d(t)}\int_{-\infty}^{-(\alpha-c(\delta))u}e^{-\frac{x^{2}%
}{2d^{2}(t)}}dx\label{alpha}\\
&  \leq Ce^{-\frac{(\alpha-c(\delta))^{2}u^{2}}{2d^{2}(t)}}.\nonumber
\end{align}
Finally we get the restriction on  $\chi$ 
for the cases (ii, iii):
\begin{equation}
\chi<\min_{t\in M(\delta)}\left(  \frac{1}{\sigma^{2}(t)}+\frac{(\alpha
-c(\delta))^{2}}{d^{2}(t)}\right)  -1.\label{chi11}%
\end{equation}
(recall that  $\sigma^{2}(t_{0})=1).$

Let us remark that $\alpha$ can be taken as small as needed. In fact, one can take \(\delta\) small enough and use that $r_{10}(t,t)\rightarrow0$ as $t\rightarrow t_{0}.$

For $x\in\lbrack-\alpha u,0]$ we consider the conditional probability under the integral  (\ref{est3n}). We get for $s>t$ and some (random) point  $t_{s}\in\lbrack t,s],$
\[
X(s)=X(t)+X^{\prime}(t)(s-t)+\frac{1}{2}X^{\prime\prime}(t_{s})(s-t)^{2}.
\]
Denote for brevity
\[
\mathcal{X}:=\{X(t)=u,X^{\prime}(t)=x\}.
\]
The conditional probability in (\ref{est3n}) is equal to
  \begin{multline} P\Bigl\{ \max_{s\in\lbrack t,b]} \bigl( X(t)+X^{\prime}(t)(s-t)+\frac{1}{2}%
X^{\prime\prime}(t_{s})(s-t)^{2}
\bigr) \geq u|\mathcal{X}\Bigr\}\nonumber \\
  =P\Bigl\{ 
  \max_{s\in\lbrack t,b]}
 \bigl( x+\frac{1}{2}%
X^{\prime\prime}(t_{s})(s-t)\bigr)\geq0|\mathcal{X}
\Bigr\}
  \leq P\Bigl\{ 
  \max_{s\in\lbrack t,b]}X^{\prime\prime}(s)\geq0\ |\mathcal{X}
  \Bigr\}%
,\label{est4.}%
\end{multline}
where we used that   $t_{s}\in\lbrack t,b]$ and $x\leq0.$ 
Denote by  $v^{2}(s;t)$ and $m(s,u,x;t)$ the conditional variance and conditional mean of the process $X^{\prime\prime}%
(s)$. The last probability can be represented as 
\begin{equation}
P\left\{  \left.  \exists s\in\lbrack t,b]:\frac{X^{\prime\prime
}(s)-m(s,u,x;t)}{v(s;t)}\geq-\frac{m(s,u,x;t)}{v(s;t)}\right\vert
\mathcal{X}\right\}.\label{est55}%
\end{equation}
Let us find a lower bound for \(- m(s,u,x;t) / v(s;t).\) Taking into account that \(\delta\) is small enough, we will simplify the expressions for the conditional mean and conditional variance as $s=t=t_{0}.$ Since $r(t_{0},t_{0})=1,$ $r_{10}(t_{0},t_{0})=0,$ we get
\begin{eqnarray*}
v^{2}(t_{0};t_{0})&=&r_{22}(t_{0},t_{0})-r_{02}^{2}(t_{0},t_{0})-\frac
{r_{12}^{2}(t_{0},t_{0})}{r_{11}(t_{0},t_{0})},\\ 
m(t_{0},u,x;t_{0}) &=& r_{02}(t_{0},t_{0})u+\frac{r_{12}(t_{0},t_{0})}%
{r_{11}(t_{0},t_{0})}x.
\end{eqnarray*}
It would be a worth mentioning that  $r_{02}(t_{0},t_{0})<0$, because 
\[
\left(  \sigma^{2}(t_{0})\right)  ^{\prime\prime}=2(r_{11}(t_{0},t_{0}%
)+r_{02}(t_{0},t_{0}))\leq0,
\]
and  $r_{11}(t_{0},t_{0})>0$ due to~\eqref{der}. Note that 
\[
-\frac{m(t_{0},u,x;t_{0})}{v(t_{0};t_{0})}=\frac{u}{v(t_{0};t_{0})}\left(
-r_{02}(t_{0},t_{0})-\frac{r_{12}(t_{0},t_{0})}{r_{11}(t_{0},t_{0})}\frac
{x}{u}\right)  .
\]
Taking into account that $|x|\leq\alpha u,$ let us choose $\delta$ in (\ref{c(de)}) small enough to guarantee the existence of  $\alpha$ such that 
\[
\frac{|r_{12}(t_{0},t_{0})|\alpha}{r_{11}(t_{0},t_{0})}\leq-\frac{1}{2}%
r_{02}(t_{0},t_{0}).
\]
Note that for \(\delta\) small enough
\begin{equation}
\min_{s,t\in\lbrack a,b]}-\frac{m(s,u,x;t)}{v(s;t)}\geq-\frac{1}{2}%
\frac{r_{02}(t_{0},t_{0})u}{v(t_{0};t_{0})}.\label{gamma11}%
\end{equation}
Due to the  Borel-TIS inequality (Theorem~2.1.1 from~\cite{AT}), the probability in  (\ref{est55}) doesn't exceed 
\begin{multline*}
P\left(  \left.  \max_{s\in\lbrack t,b]}\frac{X^{\prime\prime}%
(s)-m(s,u,x)}{v(s;t)}\geq-\frac{ur_{02}(t_{0},t_{0})}{4v(t_{0},t_{0}%
)}\right\vert \mathcal{X}\right)  \\
  \leq C\exp\Bigl\{  -\frac{1}{2}\bigl(  \frac{u|r_{02}(t_{0},t_{0}%
)|}{4v(t_{0},t_{0})}-c\bigr)^{2}\Bigr\},
\end{multline*}
where the constant $c$ doens't depend on $u.$
Finally, we conclude that 
\begin{equation}
\chi<\min\{\delta,\alpha^{2},\gamma_{2}\}.\label{chi_final}%
\end{equation}
where 

\begin{itemize}
\item $\delta$ is defined by (\ref{opt1});

\item $\alpha$ is defined by  (\ref{chi2}), (\ref{chi11});
\item \(\gamma_2\) is defined by 
\begin{equation}
\gamma_{2}:=\frac{r_{02}^{2}(t_{0},t_{0})}{16v^{2}(t_{0},t_{0})}%
.\label{gamma13}%
\end{equation}
\end{itemize}
Next, the first probability in the right-hand side of 
(\ref{est1}) has the same exponential order as $u\rightarrow\infty$  as the mathematical expectation of the number of up-crossings. This fact combined with the equality 
 (\ref{estp_1}) and  the further lines of reasoning, yields   (\ref{pmax00}).

  In the case, when the point of maximum of variance lies inside the interval $[A,B]$ or coincides with the point \(B\), the first term in  (\ref{est1}) has the same order as the remainder, provided~\eqref{der} holds. Indeed, following the same lines, we get 
 \begin{eqnarray}
\pp(u)=2\P\Bigl\{X(a%
)\geq u\Bigr\}+2\E \left[ N_{u}^{+}(\M(\delta))\right]\label{pmax00}.%
\end{eqnarray}
Since for some \(C_1>0\),
  \begin{eqnarray*}
\E \bigl[
N^+_u(\M(\delta)) 
\bigr]
\sim \frac{C_1}{u} e^{-u^2/(2S)}, \qquad u \to \infty, 
\end{eqnarray*}
 see \cite{PS},  and
  \begin{eqnarray*}
\qquad \P\{X(a) >u\} \sim \frac{1}{\sqrt{2 \pi}  \sigma(a) u} e^{-u^2/(2\sigma^2(a))}, \qquad u \to \infty,
\end{eqnarray*}
 we arrive at  (\ref{pmax}). 
 \begin{remark}
It would be a worth mentioning that our proof of~(\ref{pmax}) essentially differs from the proof of the Rice formula in~\cite{Piterbarg}, \cite{Piterbarg20}. In our proof, the existence of the second factorial moment is not used, and this fact allows us to consider milder conditions  on the process \(X(t)\).
\end{remark} 
\begin{remark}
 Let us recall that we used the assumption \(S=1.\) For general \(S\), the expression for \(\chi\) remains the same, but the computations should be repeated for the level 
 $u/\sqrt{S}$ of the process $X(t)/\sqrt{S}.$ 
 \end{remark}

Finally, let us mention that the proof of Corollary~\ref{cc} is based on the trivial inequality
\[
\P\Bigl\{ \max_{t\in \M_{1}}X(t)\geq u,\ \max_{t\in \M_{2}}X(t)\geq u
\Bigr\} 
\leq
\P\Bigl\{
\max_{(s,t)\in \M_{1}\times \M_{2}}\bigl( X(s)+X(t) \bigr)\geq2u
\Bigr\},
\]
and further application of the ideas presented in  Theorem~8.1 from \cite{Piterbarg}.

\subsection{Proof of Proposition~\ref{prop1}}\label{App1}
The idea of the proof presented below was used in our paper~\cite{KonakovPanov}. There exists also an extended version of this paper, published as a preprint~\cite{panov2014a}, where some issues are discussed in more details.

The proof consists of 6 steps. For any function \(G\) and any positive number \(h\) (probably depending on \(n,M,J,x\)), we define 
\begin{eqnarray*}
\L_{n, M, J} (h, G; x) = \L (h, G; x)
 := h  \sum_{m=1}^M \sum_{j=0}^J 
\left[ 
\int_{I_m} \psim(u) dG(u) 
\right]
 \psim(x).
\end{eqnarray*}

\textbf{1. } \textit{Koml{\'o}s-Major-Tusnady construction.}
Denote the empirical process of a uniform on \([0,1]\) random sample \(F(X_1),..,F(X_n)\) by 
\[	
	\RR_{n}(x) := \sqrt{n} \Bigl( 
		 \frac{1}{n} \sum_{i=1}^{n }\I \left\{ 		 	F( X_i)
 \leq x
		\right\} - 
x
	\Bigr).
\]
Note that  \[
\RR_{n}(F(x)) = \sqrt{n} \left(\hat{F}_n(x) - F(x) \right),\] where \(\hat{F}_n(x) = n^{-1} \sum_{k=1}^n \I\{ X_i \leq x\} \)
is the empirical distribution function.

Due to the well-known Koml{\'o}s-Major-Tusnady construction (see \cite{KMT}),  there exists a version of \(\RR_n(x)\) (denoted below also by \(\RR_n(x)\) for the sake of simplicity), a sequence of Brownian bridges \(\mathcal{B}_n(x)\), and  some positive constants \(c_1, c_2, \kappa\) such that 
\begin{eqnarray}
\label{Wstar}
\P \left\{
		\sup_{y \in [0,1]} \left| 			\RR_{n}(y) - \B_{n}(y) 
		\right| 
		\leq
 c_{1} \frac{\log(n)}{\sqrt{n}}
	\right\}
	> 
1 - \frac{
   c_2
}{
  n^{\kappa}
}.
\end{eqnarray}
Note also that under our assumptions  \(F\) is a continuous function, and therefore the event in the left-hand side of \eqref{Wstar}  is equal to  
\begin{eqnarray*}
\Omega^{*}_{n}:=
\left\{
		\sup_{x \in \R} \left| 			\RR_{n}(F(x))- \B_{n}(F(x)) 
		\right| 
		\leq
 c_{1} \frac{\log(n)}{\sqrt{n}}
	\right\}.
\end{eqnarray*}
In what follows, we denote 
\begin{eqnarray*}
\L_1(x) = \L(h, \RR_n(F(\cdot)); x), 
\qquad 
\L_2(x) = \L(h, \B_n(F(\cdot)); x).
\end{eqnarray*}
%
\textbf{2.} \textit{Supremum of the functional \(\L\); \(\L_1 \too \L_2\).} 
Let us show that 
\begin{eqnarray}
\label{w}
	\sup_{x \in [A,B]} \left| \L_{n, M, J} \Bigl(h, G; x \Bigr) \right|
	\leq 
	 c_3 h \;M  \;w(G, [A,B], \delta),
\end{eqnarray}
where
\begin{itemize} 
\item  \(w\) is the modulus of continuity, i.e., 
\begin{eqnarray*}
 	w (G, D, \delta) := \sup\Bigl\{ 
		\left| 
			G(u) - G(v)
		\right| : \; 
		u,v \in D, \; |u-v| < \delta
	\Bigr\};
\end{eqnarray*}
\item the constant \(c_3\) is equal to
\[c_3 = \sum_{j=0}^J \Bigl[
d_j + e_j (B-A)^{\alpha}
\Bigr] d_j,
\]
where \(e_j=|\psi_j|_{\alpha}\) is the \(\alpha\)-H\"older coefficient  of the function \(\psi_j,\) and \(d_j = \max_{x \in [A,B]}(\psi_j(x))\).
\end{itemize}
In fact, for \(x \in [A,B]\) there exists  an interval \(I_{m}\) containing \(x,\) and
\begin{eqnarray*}
	\left|
		\L_{n, M, J}(h,G; x) 
	\right|
	&=& 
	\left|
	h \sum_{j=0}^{J} 
	\left[
		\int_{I_m}	\psim (u) \; dG\left(u\right)
	\right] 
	\psim (x)
	\right|
		\\
	&=&
	\left| 
	h \sum_{j=0}^{J} \Bigl[
		\psim (b_m) \Bigl(
			G(b_m) - G (a_m) 
		\Bigr) \Bigr. \right.
		\\&&
		\Bigl.\left.
		\hspace{1cm}
		-
		\int_{I_m} 
			\Bigl(
   G(u)-G(a_m)\Bigr) 
		d\psim(u)
		\Bigr]
		\psim(x)
		\right|\\
		&\leq&
		h \sum_{j=0}^{J} 
		\left( 
			\sup_{x \in I_{m}}|\psim (x) |
			+ 
			V_{I_m} (\psim)
		\right)
		\sup_{x \in I_{m}}|\psim (x) | \\ && \hspace{6cm} \cdot
		w(G, [A,B], \delta)\\
		&\leq& 
		c_3 h \;M \;w(G, [A,B], \delta),
\end{eqnarray*}
where  we denote  by \(V_{I_m}
	(
		\psim
	)
	\) the total variation of \(\psim\), 
	\[
	V_{I_m}
	(
		\psim
	)
	:= \sup_{
			\|P\| \to 0
		} 
		\sum_{i=1}^{n} 
			\left| 
				\psim (x_{i}) - \psim(x_{i-1})
			\right|,\]
\(P\) ranges over the partitions \(a_m=x_{0}<x_{1}<...<x_{n}=b_m\), and  \(\|P\|=\max_{i}|x_{i}-x_{i-1}|\). In fact, for any   \(j \in \N\) and any \(m=1..M,\)
\begin{eqnarray*}
\sup_{x \in I_m} |\psi_j^{(m)}(x)| \leq  M^{1/2} d_j,\qquad
V_{I_m} (\psim) \leq M^{1/2} V_{[A,B]} (\psi_j)\leq  M^{1/2} e_j (B-A)^{\alpha}.
\end{eqnarray*}
From~\eqref{w} it follows that on the event \(\Omega_n^*\),
\begin{eqnarray*}
\left\| \L_1 - \L_2 \right\| 
&=& 
\sup_{x \in \R} \left| \L(h, \RR_n(F(\cdot))-\B_n(F(\cdot)); x)
 \right|\\
 &\leq&
c_4\; h \;M  \frac{\log(n)}{\sqrt{n}}
\end{eqnarray*}
with \(c_4 = 2 c_1 c_{3}.\)

\textbf{3.} \textit{Brownian bridge $ \too $ Brownian motion.}
Denote by \(\W_n\) the Brownian motion corresponding to the Brownian bridge \(\B_n,\) that is, \(\B_n(x) = \W_n(x) - x \W_n(1), \;\; \forall x \in [0,1].\)  Denote \(\L_3(x) = \L(h, \W_n(F(\cdot)); x).\)
\begin{eqnarray*}
\left\| 
  \L_2 -  \L_3
 \right\| &\leq & 
c_4\; h \;M  | \W_n(1) | w(F, [A,B], \delta).
\end{eqnarray*}
Since \(p \in \PP_{q,H}\),  we get
\[w(F(\cdot), [A,B], \delta) \leq  \sup_{|u-v| \leq \delta}
\int_u^v p(x) dx 
\leq 
\sup_{|u-v| \leq \delta}
\int_u^v 
\bigl(
p(u) + H \delta^\beta
\bigr) dx
\lesssim \delta,\]
where we use that \(p(u)\) is uniformly bounded for all \(p \in \PP_{q,H}.\) In fact, denote \(x_\star = \argmin_{x \in [A,B]} p(x);\) since \(p(x_\star) \leq (B-A)^{-1},\) we get from the property of uniform H{\"o}lder continuity
\begin{eqnarray*}
p(u) \leq H|u-x_\star|^\beta+p(x_\star) \leq (H+1)(B-A).
\end{eqnarray*}
Finally,
\begin{eqnarray*}
\P\left\{
\| \L_2 -  \L_3 \| \leq c_5\; h \; x
\right\} \geq \P \left\{ 
|\W_n(1)| \leq x 
\right\} = 1- 2(1-\Phi(x))
\end{eqnarray*}
for any \(x>0\) and some \(c_5>0.\)

\textbf{4.} \textit{$\L_3 \too \L_4(x):= \L(\sqrt{p(x)} h,  \W_n; x).$}
Let us first note that 
\begin{eqnarray*}
\Bigl( 
\W_n \bigl( 
F(u)
\bigr) 
\Bigr)_{u \in \R}
\eqd
\Bigl(
\int_{-\infty}^u \sqrt{p(v)} d\W_n(v)
\Bigr)_{u \in \R},
\end{eqnarray*}
because both processes are Gaussian with zero mean and covariance function \(K(u_1,u_2) = F(\min(u_1, u_2)).\)  Therefore, 
\begin{eqnarray*}
\L_4(x) - \L_3(x)  = h \sum_{m=1}^M \sum_{j=0}^J 
\left[ 
\int_{I_m} \psim(u) \bigl(
\sqrt{p(x)} - \sqrt{p(u)}
\bigr) 
  dW_n(u) 
\right]
 \psim(x).
\end{eqnarray*}
Using the integration by parts formula for Wiener integrals, we get with fixed \(m\)
\begin{multline}\label{ibp}
\Bigl|
\int_{I_m} \psim(u) \bigl(
\sqrt{p(x)} - \sqrt{p(u)}
\bigr) 
  dW_n(u) 
  \Bigr|
  \\=
  \Bigl|
  \psim(b_j) \bigl(
\sqrt{p(x)} - \sqrt{p(b_j)}
\bigr) 
W_n(b_j)
-
  \psim(a_j) \bigl(
\sqrt{p(x)} - \sqrt{p(a_j)}
\bigr) 
W_n(a_j) 
\Bigr.
\\-
\Bigl.
\int_{I_m} W_n(u) d\Bigl(
\psim(u)\left( \sqrt{p(x)} - \sqrt{p(u)}\right)
\Bigr)
\Bigr|\\
\leq 
\sup_{I_m}|W_n(u)| \cdot V_{I_m}\left(
\psim(\cdot) 
\bigl( \sqrt{p(\cdot)} - \sqrt{p(x)}\bigr)\right).
\end{multline}
Note that under our assumptions on the function \(p\), it holds for any \(x_2, x_1 \in I\),
\begin{eqnarray*}
\left| \sqrt{p(x_2)} - \sqrt{p(x_1)}\right|
\leq
\frac{
\left|
p(x_2) - p(x_1)
\right|
}{
 \sqrt{p(x_2)} + \sqrt{p(x_1)}
}
\leq c_5 |x_2 -x_1|
\end{eqnarray*}
for some \(c_5>0.\) Moreover, for any \(u_1, u_2 \in I_m, \)
\begin{multline*}
\Bigl| 
\psim(u_2) 
\bigl( \sqrt{p(u_2)} - \sqrt{p(x)}\bigr)
-
\psim(u_1) 
\bigl( \sqrt{p(u_1)} - \sqrt{p(x)}\bigr)
\Bigr| \\
\leq
\Bigl| 
\psim(u_2) 
\Bigr|
\cdot
\Bigl| 
	\sqrt{p(u_2)} - \sqrt{p(u_1)}
\Bigr|
+
\Bigl|
\psim(u_2) - \psim(u_1) 
\Bigr|
\cdot
\Bigl| \sqrt{p(u_1)} - \sqrt{p(x)}
\Bigr|
\end{multline*} 
Therefore, 
\begin{multline*}
V_{I_m}\left(
\psim(\cdot) 
\bigl( \sqrt{p(\cdot)} - \sqrt{p(x)}\bigr)\right)
\leq 
c_5 \Bigl(
\sup_{u \in I_m} \bigl| 
	\psim(u) 
\bigr|
+
V_{I_m} \bigl(
\psim
\bigr)
\Bigr)\cdot
\bigl| I_m \bigr|
\\
\leq c_6  M^{-1/2}
\end{multline*}
with some \(c_6=c_5 (d_j + e_j(B-A)^{\alpha}) (B-A)\). Combining this result with \eqref{ibp}, we arrive at 
\begin{eqnarray*}
\left\| \L_4(x) - \L_3(x) \right\| 
  \leq 
c_7 h \sup_{[A,B]} |W_n(u)|.
\end{eqnarray*}
Next,  
\begin{eqnarray*}
	\P \Bigl\{
		\sup_{[A,B]}|W_{n}(u)| 
		>
		x
	\Bigr\} 
	&\leq& 
	\P \Bigl\{
		\sup_{[A,B]}W_{n}(u)
		>
		x
	\Bigr\} 
	+
	\P \Bigl\{
		\inf_{[A,B]}W_{n}(u) 
		<
		-x
	\Bigr\} \\
	&\leq&
		2 \P \Bigl\{
		\sup_{[A,B]}W_{n}(u)
		>
		x
	\Bigr\} 
	= 
	 4 \P \Bigl\{
	 	W_{n}(B) >x
	\Bigr\},
\end{eqnarray*}
where the last inequality follows from  the reflection principle for Brownian motion. Therefore, for any \(x>0,\)
\begin{eqnarray*}
\P\left\{
\| \L_3 -  \L_4 \| \leq c_7 \; h \; x
\right\} &\geq& \P \Bigl\{
		\sup_{[A,B]}|W_{n}(u)| 
		<
		x
	\Bigr\}\\
	&\geq& 1-4\Bigl( 
	1- \Phi(x/\sqrt{B})
	\Bigr).
\end{eqnarray*}

\textbf{6.} \textit{Final step.} Note that 
\begin{eqnarray*}
\sup_{x \in \R} \frac{
\left|
\L_4(x) 
\right|
}{\sqrt{p(x)}
}
&=& h
\max_{m=1..M} 
\sup_{x \in \R} \left|
\sum_{j=0}^J 
\left[ 
\int_{I_m} \psim(u)
  dW_n(u) 
\right]
 \psim(x)
\right|\\
&=& h \sqrt{M}
\max_{m=1..M} 
\sup_{x \in I} \left|
\sum_{j=0}^J 
\left[ 
\int_{I} \psi(u)
  dW_n^{(m)}(u) 
\right]
 \psi(x)
\right|,
\end{eqnarray*}
where \(W_n^{(m)}, m=1..M\) are i.i.d. copies of \(W_n.\)
The last formula suggests the choice \(h=M^{-1/2}.\)
To complete the proof, we need the following technical lemma, which is proven in  
\cite{panov2014a}, pp.~33-34. 

\begin{lemma}
\label{final}
	Let \(\eta_{1}, ..., \eta_{k}\) be random variables such that
	\begin{eqnarray*}
		\P \Bigl\{ 
			\left| 
				\eta_{i+1} - \eta_{i}
			\right| 
			\leq \delta_{i}
		\Bigr\} \geq 1 -  \gamma_{i}, \quad i=1..(k-1),
	\end{eqnarray*}
for some non-negative \(\delta_{i}, \gamma_{i}, \; i=1..k\).  Denote by \(F_{\eta_{k}}\) the distribution function of \(\eta_{k}\). Then  it holds
\begin{eqnarray}
	F_{\eta_{1}}\left( 
		x- \sum_{j=1}^{k-1}\delta_{j}
	\right) 
	-
	\sum_{j=1}^{k-1}\gamma_{j}
	\leq 
	F_{\eta_{k}}(x) 
	\leq 
		F_{\eta_{1}}\left( 
		x	+ \sum_{j=1}^{k-1}\delta_{j}
	\right) 
	+
	\sum_{j=1}^{k-1}\gamma_{j}.
	\label{reslem}
\end{eqnarray}
\end{lemma} 
Let us apply this lemma with 
\begin{eqnarray*}
\eta_k = 
\|\L_k(x) / \sqrt{p(x)}\|, \qquad k=1..4.
\end{eqnarray*}
For any \(k=1..4,\)
\begin{eqnarray*}
	\left| 
		\eta_{k} - \eta_{k-1}
	\right| \leq q^{-1/2}
		\Bigl| 
		\left\|
			\L_{k}(x)
		\right\| 
		-
		\left\|
			\L_{k-1}(x)
		\right\| 
	\Bigr|
\leq
q^{-1/2}
\left\| 
		\L_{k}(x) - \L_{k-1}(x)
\right\|.
\end{eqnarray*}    
Therefore, under the choice \(h=M^{-1/2},\)
\begin{eqnarray*}
		\delta_{1}=c_4 q^{-1/2}\; M^{1/2} 
 \frac{\log(n)}{\sqrt{n}}, &&\qquad  \gamma_{1} =\frac{
   c_2
}{
  n^{-\kappa}
},\\
\delta_{2}=c_4 q^{-1/2}\; M^{-1/2} \; k_1, &&\qquad  \gamma_{2} =2(1-\Phi(k_1)),\\
\delta_{3}=c_7 q^{-1/2}\; M^{-1/2}  k_2, &&\qquad  \gamma_{3} =4(1-\Phi(k_2 / \sqrt{B})),
\end{eqnarray*}
with any \(k_1,k_2>0.\)   The choice of these parameters is based on the idea to have  \(\gamma_1 \asymp \gamma_2 \asymp \gamma_3\). Taking into account the inequality 
\begin{eqnarray}
\label{Michna}
1 - \Phi(x) \leq 	 \frac{ 1}{x\sqrt{2 \pi}} e^{-x^{2}/2}, \qquad \forall \; x>0,
\end{eqnarray}
see, e.g., p.2 in \cite{Michna}, we choose \(k_1=\sqrt{2 \kappa \log n}\), and \(k_{2}=\sqrt{2 B \kappa \log n}\). Under this choice of \(q_1, q_2\), we have 
\[
\delta_1 + \delta_2 + \delta_3 \leq 
c_9 \frac{\log(n)}{\sqrt{n/M}} + c_{10} \frac{\sqrt{\log(n)}}{\sqrt{M}}.
\]\

\subsection{Proof of Theorem~\ref{thm35}}
\label{app4}
The proof consists of 4 steps. On the first step, we consider the distribution function 
\begin{eqnarray*}
F_{n,M}(x) :=
	\P \left\{
		\sqrt{\frac{n}{M}}
\RR_{n}
		\leq x	\right\},
\end{eqnarray*}
and show that it is close to \(A_M(x)\) for \(x\geq c_M\), see \eqref{amm2}. Next, on step~2, we show similar result for \(x<c_M\).  Finally, on step~3, we show that \(\RR_n(x)\) can be replaced by \(\DDD_n(x)\).

\textbf{1.}  
Let us first show that
\begin{eqnarray}
\label{amm2}
\sup_{x \geq c_M} 
\left| 
F_{n,M}(x)	- A_{M}(x)
\right| 
	\leq
\bar{c}_1\; M^{-\theta}+
	\bar{c}_2\; n^{-\kappa}.
\end{eqnarray}
From Corollary~\ref{cc} we get for any \(x\geq c_{M}\)
\begin{eqnarray*}
F_{n,M}(x)
& \leq & 
\left[
\P \Bigl\{ \max|\Upsilon(x)| \leq x+\gamma_{n,M} \Bigr\} 
\right]^{M} +  \C_1 n^{-\kappa}\\
& = & 
\exp\Bigl\{ 
M  \log
\Bigl[ 1- 
\sum_{i=1}^k \pp_i(x+\gamma_{n,M} ) 
 + r(x)
\Bigr]
\Bigr\} +  \C_1 n^{-\kappa},
\end{eqnarray*}
where \(r(x)=O(M^{-(1+\chi)}),\) uniformly over all \(x \geq c_M.\) Therefore, 
\begin{multline*}
F_{n,M}(x)
 \leq 
A_{M}(x+\gamma_{n,M})
\exp\Bigl\{-\frac{M}{2} \bigl( \sum_{i=1}^k \pp_i(x+\gamma_{n,M}) \bigr)^2 (1+o(1))+O(M^{-\chi})
\Bigr\} \\ +  \C_1 n^{-\kappa}.
\end{multline*}
From Remark~\ref{corcor}, we get for \(M\) large enough,
\begin{eqnarray*}
\sum_{i=1}^k \pp_i(x+\gamma_{n,M}) \lesssim
 \frac{1}{c_M} e^{-c_M^2/(2S)}\lesssim M^{-1} 
 \frac{ e^{(2S \log M)^{1/2}}}{(2S \log(M))^{1/2}}
\lesssim M^{-H}
\end{eqnarray*}
for any \(H \in (0,1).\) Taking \(H>1/2,\) we arrive at 
\begin{eqnarray*}
F_{n,M}(x)
& \leq & 
A_{M}(x+\gamma_{n,M}) + \bar{c}_1\; M^{-\theta}+
	\bar{c}_2\; n^{-\kappa}\end{eqnarray*}
with some \(\bar{c}_1>0\) and \(\bar{c}_2 = \C_1. \) Due to Lemma~\ref{lemapp}, \(A_M(x)\) and \(A_M(x+\gamma_{n,M})\) differ by a quantity, which uniformly converges to zero at polynomial rate w.r.t. \(n,M.\)  The proof for the inverse inequality follows the same lines.

\textbf{2.} 
Let us consider now the  case \(x<c_M\). From Proposition~\ref{prop1} and Remark~\ref{corcor} we get 
\begin{eqnarray*}
F_{n,M}(x)
&\leq& \exp\Bigl\{
M \log \P \bigl\{
\zeta \leq c_M +\gamma_{n,M}
\bigr\}
\Bigr\}+\C_1 n^{-\kappa}\\
&=&
\exp\Bigl\{
M \log \bigl( 
1 - \frac{1}{2 \pi (c_M+\gamma_{n,M})}
e^{-(c_M+\gamma_{n,M})^2/(2S)}\mathfrak{c}_0 
(1+o(1))
\bigr)
\Bigr\}\\&&\hspace{8cm}+\C_1 n^{-\kappa}\\
&=&
\exp\Bigl\{
-
\frac{\mathfrak{c}_0 }{2\pi}  \cdot 
 \frac{ e^{(2S \log M)^{1/2}}}{(2S \log(M))^{1/2}}
(1+o(1))
\bigr)
\Bigr\}+\C_1 n^{-\kappa}.
\end{eqnarray*}
Since  
\begin{eqnarray}\label{bol}
\frac{ e^{(2S \log M)^{1/2}}}{(2S \log(M))^{1/2}} > K \log M
\end{eqnarray}
for any \(K>0\) and \(M\) large enough, we get that \(F_{n,M}(x)\) converges to zero at polynomial rate with respect to both \(n\) and \(M.\)

\textbf{3.} According to~\eqref{ineq},
\begin{eqnarray*}
\P \left\{
		\sqrt{\frac{n}{M}}
\DDD_{n}
		\leq x
	\right\}
	&\leq&
	\P \left\{
\sqrt{\frac{n}{M}}
\RR_{n}	\leq x 	+	
C_4 n^{1/2} M^{-\beta-1/2} 
\right\}. 
\end{eqnarray*}
Due to the steps 1,2 of this proof, we have 
\begin{eqnarray*}
\P \left\{
		\sqrt{\frac{n}{M}}
\DDD_{n}
		\leq x
	\right\} \leq 
	A_M\Bigl(x+ C_4 n^{1/2} M^{-\beta-1/2} \Bigr)+\bar{c}_1\; M^{-\theta}+
	\bar{c}_2\; n^{-\kappa}.
\end{eqnarray*}
Applying once more Lemma~\ref{lemapp}, we arrive at 
\begin{eqnarray*}
\P \left\{
		\sqrt{\frac{n}{M}}
\DDD_{n}
		\leq x
	\right\} 
	\leq A_M(x) + c_1  n^{1/2} M^{\theta_1-\beta-1/2}  
	+\bar{c}_1\; M^{-\theta}+
	\bar{c}_2\; n^{-\kappa}
\end{eqnarray*}
with some \(c_1>0.\) Note that  if \(M=\lfloor n^\lambda \rfloor\) with \(\lambda \in ((2 \beta+1)^{-1},1)\), there exists some \(\theta_1\) such that 
\begin{eqnarray*}
 c_1  n^{1/2} M^{\theta_1-\beta-1/2}  
	+\bar{c}_1\; M^{-\theta}+
	\bar{c}_2\; n^{-\kappa}
	= \bar{c}n^{-\gamma}
\end{eqnarray*}
with some \(\bar{c}, \gamma>0.\) This observation completes the proof.
\bibliographystyle{spbasic}
\bibliography{Panov_bibliography}

\begin{thebibliography}{24}
\providecommand{\natexlab}[1]{#1}
\providecommand{\url}[1]{{#1}}
\providecommand{\urlprefix}{URL }
\expandafter\ifx\csname urlstyle\endcsname\relax
  \providecommand{\doi}[1]{DOI~\discretionary{}{}{}#1}\else
  \providecommand{\doi}{DOI~\discretionary{}{}{}\begingroup
  \urlstyle{rm}\Url}\fi
\providecommand{\eprint}[2][]{\url{#2}}

\bibitem[{{Adler, R. and Taylor, J.}(2009)}]{AT}
{Adler, R and Taylor, J} (2009) {Random fields and geometry}. Springer Science
  \& Business Media

\bibitem[{{Bai, L., D\c{e}bicki, K., Hashorva, E. and Ji, L.}(2018)}]{BDHJ}
{Bai, L, D\c{e}bicki, K, Hashorva, E and Ji, L} (2018) {Extremes of
  threshold-dependent Gaussian processes}. Science China Mathematics
  61(11):1971--2002

\bibitem[{{Bai, L., D\c{e}bicki, K., Hashorva, E., and Luo, L.}(2018)}]{BDHL}
{Bai, L, D\c{e}bicki, K, Hashorva, E, and Luo, L} (2018) {On Generalised
  Piterbarg Constants.} {Methodology \& Computing in Applied Probability} 20(1)

\bibitem[{{Bickel, P., and Rosenblatt, M.}(1973)}]{Bickel}
{Bickel, P, and Rosenblatt, M} (1973) {On some global measures of the
  deviations of density function estimates}. {The Annals of Statistics}
  1(6):1071--1095

\bibitem[{{Bull, A.}(2012)}]{Bull}
{Bull, A} (2012) {Honest adaptive confidence bands and self-similar functions}.
  {Electronic Journal of Statistics} 6:1490--1516

\bibitem[{{Chernozhukov, V. and Chetverikov, D., and Kato, K.}(2014)}]{CCK}
{Chernozhukov, V and Chetverikov, D, and Kato, K} (2014) {Anti-concentration
  and honest, adaptive confidence bands}. {The Annals of Statistics}
  42(5):1787--1818

\bibitem[{{Gin{\'e}, E., and Koltchinskii, V., and Sakhanenko, L.}(2004)}]{GKH}
{Gin{\'e}, E, and Koltchinskii, V, and Sakhanenko, L} (2004) {Kernel density
  estimators: convergence in distribution for weighted sup-norms}. Probability
  Theory and Related Fields 130(2):167--198

\bibitem[{{Gin{\'e}, E. and Nickl, R.}(2010)}]{GN10}
{Gin{\'e}, E and Nickl, R} (2010) {Confidence bands in density estimation}.
  {The Annals of Statistics} 38(2):1122--1170

\bibitem[{{Gin{\'e}, E. and Nickl, R.}(2016)}]{GN}
{Gin{\'e}, E and Nickl, R} (2016) {Mathematical foundations of
  infinite-dimensional statistical models}, vol~40. Cambridge University Press

\bibitem[{{Hashorva, E. and H{\"u}sler, J.}(2000)}]{HH2000}
{Hashorva, E and H{\"u}sler, J} (2000) {Extremes of Gaussian processes with
  maximal variance near the boundary points}. {Methodology and Computing in
  Applied Probability} 2(3):255--269

\bibitem[{{H{\"u}sler, J and Piterbarg, V}(2004)}]{HP}
{H{\"u}sler, J and Piterbarg, V} (2004) {On the ruin probability for physical
  fractional Brownian motion}. {Stochastic Processes and their Applications}
  113(2):315--332

\bibitem[{{H{\"u}sler, J. and Piterbarg, V. and Seleznjev, O.}(2003)}]{HPS}
{H{\"u}sler, J and Piterbarg, V and Seleznjev, O} (2003) {On convergence of the
  uniform norms for Gaussian processes and linear approximation problems}. The
  Annals of Applied Probability 13(4):1615--1653

\bibitem[{{Koml{\'o}s, J., Major, P., and Tusn{\'a}dy, G.}({1975})}]{KMT}
{Koml{\'o}s, J, Major, P, and Tusn{\'a}dy, G} ({1975}) {An approximation of
  partial sums of independent rv's and the sample DF.} {Zeitschrift f{\"u}r
  Wahrscheinlichkeitstheorie und Verw Gebiete} {32}:{111--131}

\bibitem[{{Konakov, V. and Panov, V.}(2016)}]{KonakovPanov}
{Konakov, V and Panov, V} (2016) {Sup-norm convergence rates for L{\'e}vy
  density estimation}. Extremes 19(3):371--403

\bibitem[{{Konakov,V., and Panov,V.}(2016)}]{panov2014a}
{Konakov,V, and Panov,V} (2016) {Convergence rates of maximal deviation
  distribution for projection estimates of L{\'e}vy densities}.
  {arXiv:1411.4750v3}

\bibitem[{{Konstant, D. and Piterbarg, V.}(1993)}]{KP93}
{Konstant, D and Piterbarg, V} (1993) {Extreme values of the cyclostationary
  Gaussian random process}. Journal of applied probability 30(1):82--97

\bibitem[{{Marron, J. and Wand, M.}(1992)}]{MW}
{Marron, J and Wand, M} (1992) {Exact mean integrated squared error}. {The
  Annals of Statistics} pp 712--736

\bibitem[{{Michna, Z.}(2009)}]{Michna}
{Michna, Z} (2009) {Remarks on Pickands theorem}. {Arxiv: 0904.3832v1}

\bibitem[{{Piterbarg, V.}(1996)}]{Piterbarg}
{Piterbarg, V} (1996) {Asymptotic methods in the theory of Gaussian processes
  and fields}. {AMS, Providence}

\bibitem[{{Piterbarg, V.}(2015)}]{Piterbarg20}
{Piterbarg, V} (2015) {Twenty lectures about Gaussian processes}. {Atlantic
  Financial Press, London, New York}

\bibitem[{{Piterbarg, V. and Prisiazhniuk, V.}(1978)}]{PP}
{Piterbarg, V and Prisiazhniuk, V} (1978) {Asymptotic analysis of the
  probability of large excursions for a nonstationary Gaussian process}.
  {Teoriia Veroiatnostei i Matematicheskaia Statistika} 18:121--134

\bibitem[{{Piterbarg, V. and Simonova, I.}(1984)}]{PS}
{Piterbarg, V and Simonova, I} (1984) {Asymptotic expansions for the
  probabilities of large runs of nonstationary Gaussian processes}.
  {Mathematical Notes} 35(6):477--483

\bibitem[{{Smirnov, N.V.}(1950)}]{Smirnov}
{Smirnov, NV} (1950) {On the construction of confidence regions for the density
  of distribution of random variables}. In: {Doklady Akad. Nauk SSSR}, vol~74,
  pp 189--191

\bibitem[{{Wasserman, L.}(2006)}]{W2006}
{Wasserman, L} (2006) {All of nonparametric statistics}. Springer Science \&
  Business Media

\end{thebibliography}

\appendix

\section{Choice of the parameter $\delta$} \label{A}
In this section, we provide an example on the choice of the parameter \(\delta\) in Theorem~\ref{thm2}. We will concentrate on the case of the Gaussian process\begin{eqnarray}\label{Ups2}
\Upsilon(t) = \sum_{j=0}^J \psi_j (t) Z_j,
\end{eqnarray}
where \(Z_0, Z_1,...\) are the Legendre polynomials. As it is explained in Section~\ref{exmain}, this case corresponds to  the item (i) in Theorem~\ref{thm2}:  $t_{0}=A=-1,$ $r_{10}(t_{0},t_{0})=\frac
{1}{2}(\sigma^{2}(t_{0}))^{\prime}<0$ (the behaviour at the point $t_{0}=B=1$ is completely the same). Denote 
\[
D(\delta):=\max_{t\in \widetilde{M}(\delta)}\bigl[-(\sigma^{2}(t))^{\prime}\bigr],
\]
where \(\widetilde{M}(\delta) =\M(\delta) \cap [-1,0]\). In what follows, we assume that \(\delta\) is such that  \(\wM(\delta) =[A,b]\) for some \(b \in (-1,0).\) In the considered case, the absolute value of \((\sigma^{2}(t))^{\prime}\) decays in some right vicinity of the point \(t_0=-1\) (for instance, the plot of \(\sigma^{2}(t)\) for \(J=4\) is given on Figure~\ref{Fig2}) and therefore we can take \(\delta\) such that 
\begin{equation}
D(\delta)=-(\sigma^{2}(A))^{\prime}=2|r_{10}(A,A)|. \label{1}%
\end{equation}
\begin{figure*}
\begin{center}
\includegraphics[width=0.75\linewidth ]{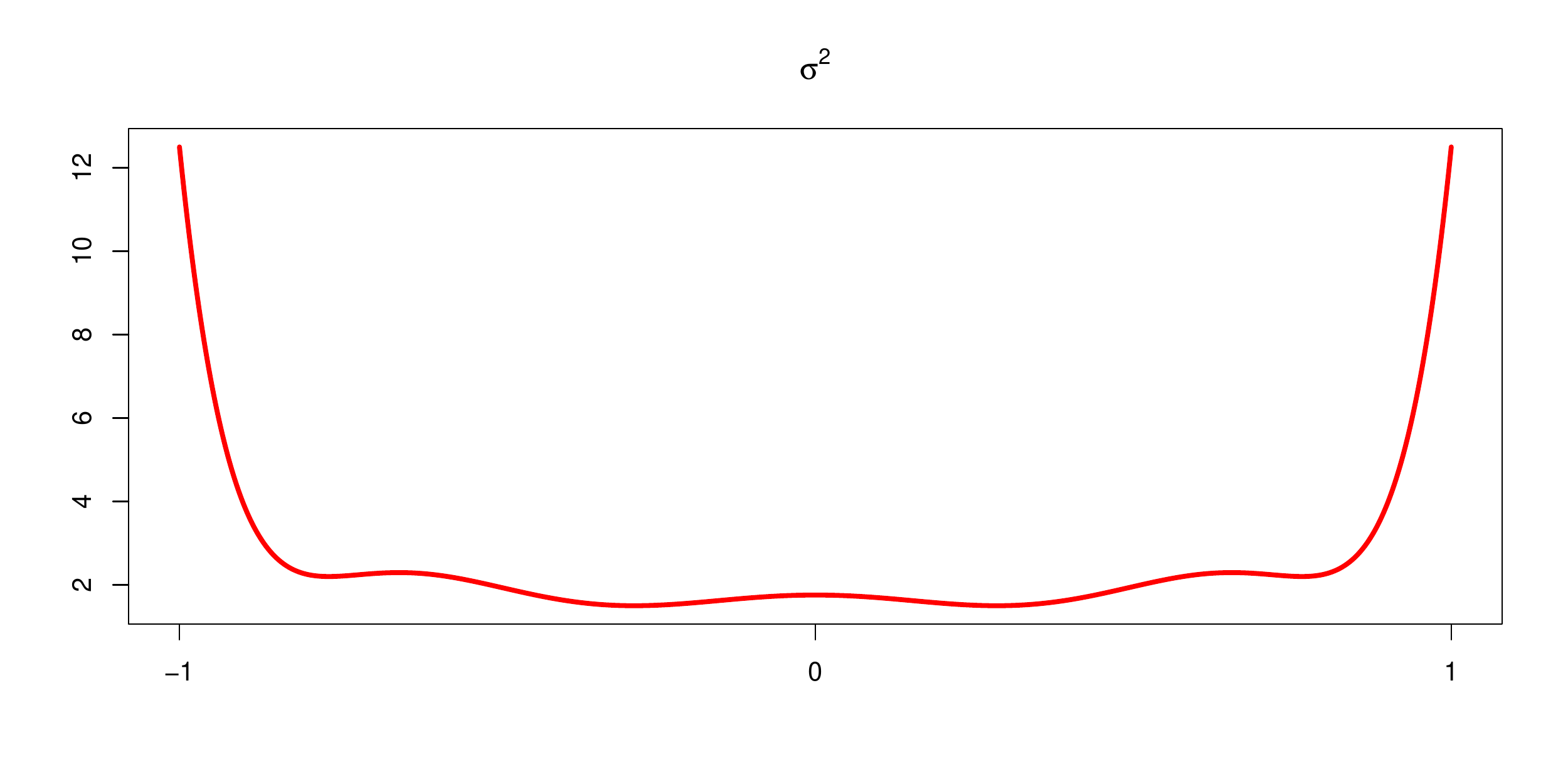}\caption{ Plot of the variance \(\sigma^2(t) = \Var(\Upsilon(t))\) for \(J=4.\)  Note that \(S= \max\sigma^2(t) =(J+1)^2/2=25/2\), the set \(\M(\delta):=\{t: \sigma^2(t)>S/(1+\delta)\}\) is an interval for any \(\delta<4.44.\)} \label{Fig2}
\end{center}
\end{figure*}

Now we aim to find a lower bound for  $\min_{s,t\in\wM(\delta)}r(s,t)$.  The two-dimensional mean value theorem yields
\begin{multline*}
r(s,t) =r(A,A)  +\int_{0}^{1}\Bigl[(s-A)r_{10}(A+h(s-A),A+h(t-A))
\\+(t-A)r_{01}%
(A+h(s-A),A+h(t-A))\Bigr]dh.
\end{multline*}
We get 
\[
r(s,t)\geq S-2(b-A)D_{1}(\delta),\qquad s,t\in \wM(\delta),
\]
where  
\(
D_{1}(\delta)=\max_{s,t\in \wM(\delta)}(-r_{01}(s,t)).
\)
The inequality (\ref{chi0}) reads as 
\begin{equation}
\chi<\chi_{1}(\delta):=\min\left\{  \delta,\frac{S-(b-A)D_{1}(\delta)}{(b-A)D_{1}(\delta)}\right\}.\label{chiex1}%
\end{equation}
Note that for the Legendre polynomials, \(A=-1\) and
\begin{eqnarray}\label{D1}
D_1(\delta) = \max_{s,t\in \wM(\delta)} \Bigl( - \sum_{j=0}^J \psi_j(s) \psi'_j(t) \Bigr)= - \sum_{j=0}^J \psi_j(-1) \psi'_j(-1),
\end{eqnarray}
because  for  \(\delta\) small enough it holds 
\begin{eqnarray} \label{prop11} -  \psi_j(s) \psi'_j(t) \geq 0 \qquad \forall s,t \in \wM(\delta), \; \; \forall j=0,1,..
\end{eqnarray} and \[\argmax_{s \in  \wM(\delta)} |\psi'_j(s)| = \argmax_{t \in  \wM(\delta)} |\psi_j(t)| = -1,\]
see, e.g., Section~5.1 from~\cite{panov2014a}. The last expression in~\eqref{D1} and \(S\) can be directly computed:
\begin{eqnarray*}
- \sum_{j=0}^J \psi'_j(-1) \psi_j(-1) = 
\frac{1}{4} \sum_{j=0}^J j(j+1)(2j+1)&=&\frac{1}{8}J(J+1)^2(J+2),\\
S = \sum_{j=0}^J \psi_j^2(-1)=\sum_{j=0}^J \frac{2j+1}{2} &=& \frac{(J+1)^2}{2}.
\end{eqnarray*}
Therefore, we conclude that 
\begin{eqnarray*}
\chi_1(\delta) = \min\left\{  \delta,\frac{4}{(b-A)J(J+2)} - 1 \right\}
\end{eqnarray*}
The second restriction on \(\chi\) arrises due to 
 (\ref{chi00}):
\begin{equation}
\chi\leq\chi_{2}(\delta):=\min_{t\in \wM(\delta)} \Psi(t), \qquad \Psi(t) := \frac{r_{10}^{2}(t,t)}%
{r_{11}(t,t)}.\label{chiex2}%
\end{equation}
This function cann't be simplified in the considered particular case, and will be analysed numerically later. 

The last restriction on \(\chi \) appears due to the Corollary~\ref{cc}.
Applying Theorem~8.1 from \cite{Piterbarg}, we get 
\begin{eqnarray}\label{rest}
P_{2u}(X(s)+X(t),\wM(\delta) \times ( \M(\delta)\setminus \wM(\delta)))\leq 
C(b-A)^{2}ue^{-2u^{2}/R(\delta)}
\end{eqnarray}
with some \(C>0\) and  \[R(\delta) =\max_{
\begin{smallmatrix} 
s\in  \wM(\delta), \\ t \in \M(\delta)\setminus \wM(\delta)
\end{smallmatrix}} \Var(\Upsilon_s+\Upsilon_t).\] 
We have
\[
\Var (\Upsilon_s+\Upsilon_t)
=r(s,s)+r(t,t)+2r(s,t) = 
\sum_{j=0}^J \bigl( \psi_j (s) + \psi_j(t) \bigr)^2.
\]
For the Legendre polynomials, it holds \(\psi_j(t) = (-1)^j \psi_j(-t)\), and therefore 
\begin{eqnarray*}
R(\delta) = \max_{s,t \in \wM(\delta)} \RR(s,t), \qquad \mbox{where} \quad \RR(s,t):=\sum_{j=0}^J \bigl( \psi_j (s) + (-1)^j \psi_j(t) \bigr)^2. 
\end{eqnarray*}
Empirically, we get that the maximum of \(\RR(s,t)\) is attained at the point \((-1,-1)\) (see Figure~\ref{Fig4}), and therefore
\begin{eqnarray*}
R(\delta) = \begin{cases}
(J+1)(J+2), & J \; \; \mbox{is even},\\
J(J+1), & J \; \; \mbox{is odd},
\end{cases}
\end{eqnarray*}
for any \(\delta\), which guarantees that the set \(\M(\delta)\) is an interval.
Therefore, from~\eqref{rest} we get the last restriction on \(\chi:\)
\begin{equation}
\chi<\chi_{3}(\delta):=\frac{4S}{R(\delta)}-1
=
\begin{cases}
J/ (J+2), & J \; \; \mbox{is even},\\
(J+2)/J, & J \; \; \mbox{is odd}.
\end{cases}\label{chiex3}%
\end{equation}
Finally, we conclude that for the optimisation of \(\chi\), we should find as follows  
\[
\chi_{opt}=\max_{\delta} m(\delta), \qquad \mbox{where} \quad m(\delta):=\min\{\chi_{1}(\delta),\chi_{2}(\delta),\chi_{3}(\delta)\}.
\]
This optimisation procedure is illustrated on Figure~\ref{Fig3}. The left picture presents the plots of the functions  \(\chi_{1}(\delta),\chi_{2}(\delta),\chi_{3}(\delta)\), while the right picture depicts the minimum between them. It turns out, that the maximum over \(\delta\) is equal to \(2/3,\) and this value is attained for any \(\delta \in (0.71, 2.13).\) 

\begin{figure*}
\begin{center}
\includegraphics[width=0.8\linewidth ]{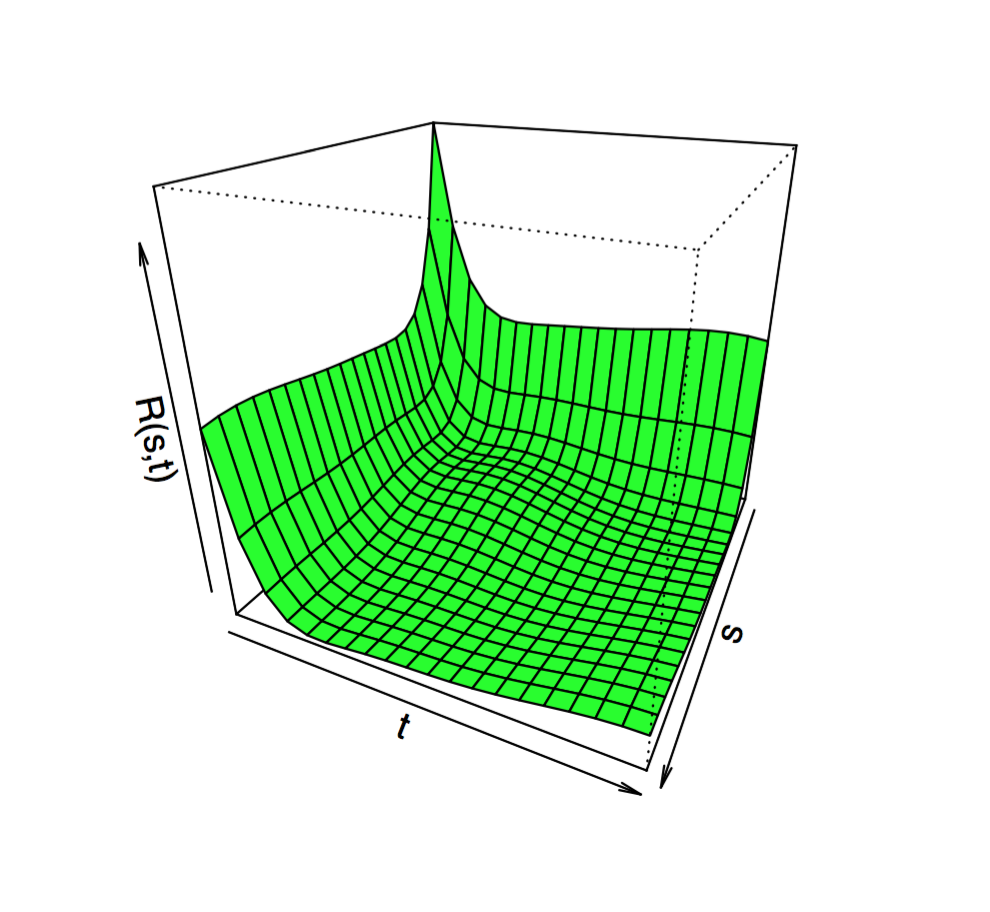}\caption{\label{Fig4} Plot of the function \(\RR(s,t)\) for \(J=4.\) Maximum is attained at the point \((-1,-1).\)}
\end{center}
\end{figure*}

\begin{figure*}
\begin{center}
\includegraphics[width=1\linewidth ]{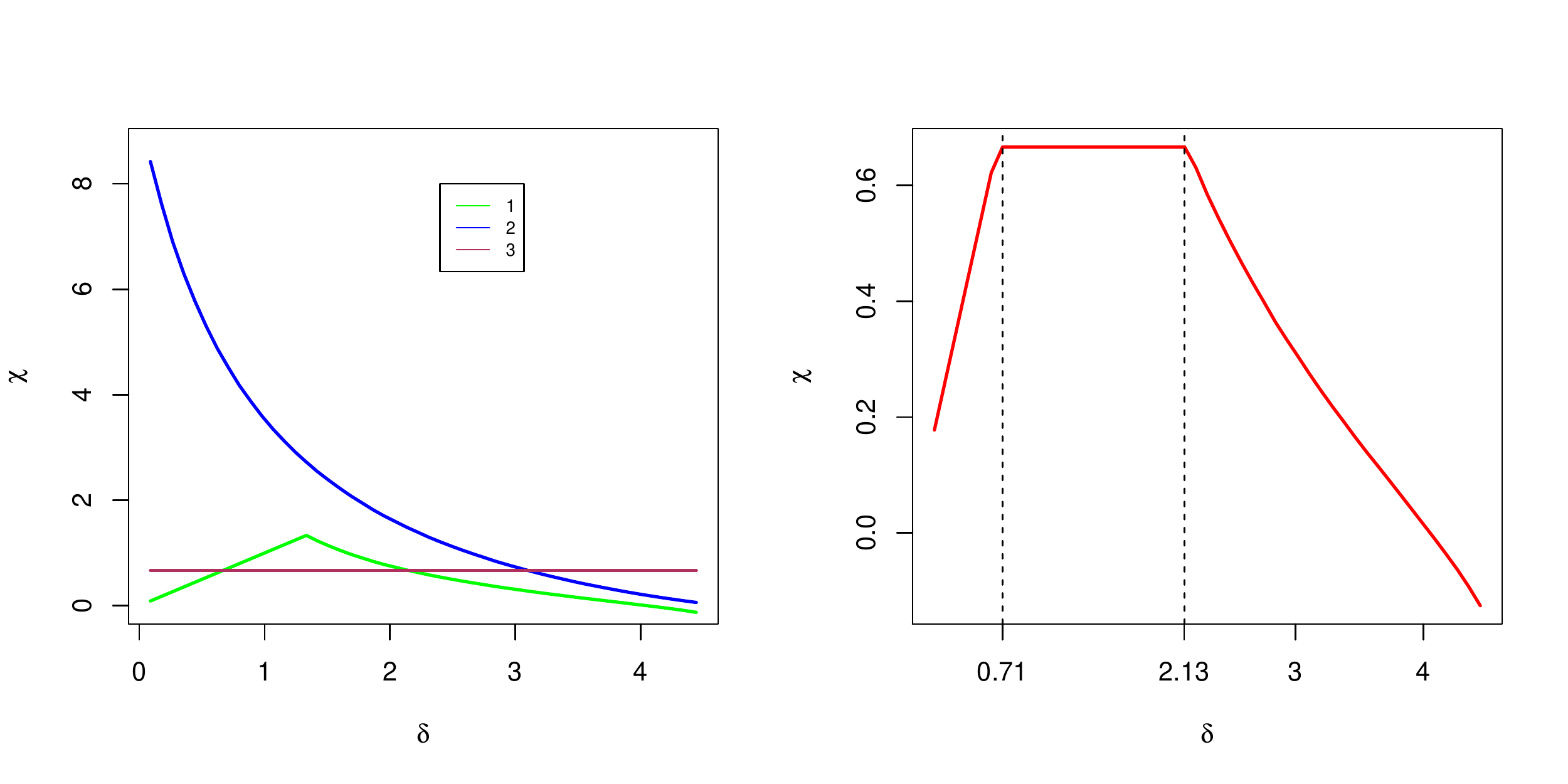}\caption{\label{Fig3} Plots of the function \(\chi_1(t), \chi_2(t), \chi_3(t)\) (left) and of the function \(m(\delta)=\min(\chi_1(t),\chi_2(t),\chi_3(t))\) (right) for \(J=4.\)    The maximal value of \(m(\delta)\) is equal to \(J/(J+2)=2/3, \) and this value is attained with any \(\delta \in (0.71, 2.13).\)}
\end{center}
\end{figure*}




\section{SBR-type theorem for projection density estimates}\label{SBR}
In this subsection, we briefly discuss the SBR-type theorem for the estimate~\eqref{projest}. The next theorem shows that the distribution of \(\RR_{n}\) converges to the Gumbel distribution. Nevertheless, the rate of this convergence is very slow, of logarithmic order. 
\begin{theorem}\label{thm3}
Assume that \(p\in \mathcal{P}_{q,H,\beta}\) with some \(q,H>0, \; \beta \in (0,1]\), and the basis functions \(\psi_j(x),  j=0,1,2,...\) satisfy the assumptions (A1) and (A2).  Let \(M=M_n=\lfloor n^\lambda \rfloor\) with \(\lambda \in (0,1)\).

\begin{enumerate}[(i)]
\item For any \(x \in \R\), it holds
\begin{eqnarray}\label{main1}
		\P \Bigl\{\sqrt{\frac{n}{M_n}}\RR_{n} \leq u_M(x) \Bigr\} =e^{-  e^{-x}} \left( 
			1 +e^{-x}  \Lambda_M (1+o(1))
		\right),
\end{eqnarray}
as \(n \to \infty,\)
where 
\begin{eqnarray*}
\Lambda_M = \frac{\bigl( \log \log (M) \bigr)^2}{16 \log (M)} 
\end{eqnarray*}
and
\begin{eqnarray} \label{un}
u_M(x) &=& a_M + \frac{xS}{a_M}, \\
\label{un2} 
a_M  &=& \bigl( 2S \log(M) \bigr)^{1/2} - \frac{S^{1/2}}{2^{3/2}} \frac{\log\bigl(
\bigl( 8 \pi^2 S / \mathfrak{c}_0  \bigr) \log(M)
\bigr)
}
{\bigl( \log(M) \bigr)^{1/2}
}
\end{eqnarray}
with \(S=\max \sigma^2(t)\), and \(\mathfrak{c}_0\) defined by \eqref{Gauss}  for \(X(t) = \Upsilon(t)\).
\item In~\eqref{main1}, \(\RR_n\) can be changed to \(\DDD_n\), provided \(\lambda \in (1/(2\beta+1),1).\)
\end{enumerate}
\end{theorem}
\begin{proof}
\underline{\textbf{(i)}.} Note that the process \(\p(x)\) defined by~\eqref{zeta} has zero mean and variance equal to \eqref{vvar}. From \eqref{Fst1} and \eqref{Gauss}, we get for \(u \to \infty,\)
\begin{eqnarray}\nonumber
\P \Bigl\{\sqrt{\frac{n}{M}}\RR_{n} \leq u \Bigr\}
		& \leq & 
\left[
\P \Bigl\{ \max|\Upsilon(x)| \leq u + \gamma_{n,M} \Bigr\} 
\right]^{M} +  \C_{1} n^{-\kappa}\\
\nonumber
&=& 
\exp\Bigl\{M
\log
\Bigl(
1 - \frac{
1
}{2 \pi 
}
e^{-\u_{n,M}^2/(2 \S)} 
\u_{n,M}^{-1}
\bigl( 
\mathfrak{c}_0 + \mathfrak{c}_1 \u_{n,M}^{-2} + o(\u_{n,M}^{-2}) 
\bigr)
\Bigr)
\Bigr\} \\ \nonumber
&& \hspace{7cm}+ \C_{1} n^{-\kappa}
\\ \nonumber
&=&  
\exp\Bigl\{
- \frac{
M
}{2 \pi 
}
e^{-\u_{n,M}^2/(2 \S)} 
\u_{n,M}^{-1}
\Bigl( 
\mathfrak{c}_0 + \mathfrak{c}_1 \u_{n,M}^{-2} 
+ o(\u_{n,M}^{-2}) 
\Bigr)\Bigr\} 
\\
&& \hspace{5cm}
+  \C_{1} n^{-\kappa} \label{i1}
\end{eqnarray}
with \(\u_{n,M}  =u+\gamma_{n,M}\). 
Next, let us substitute  \(u = a_M + x S / a_M - \gamma_{n,M}\) with some \(a_M \to \infty\)  as \(M \to \infty\).  We have 
\begin{multline*}
\P \Bigl\{\sqrt{\frac{n}{M}}\RR_{n} \leq a_M + \frac{xS}{a_M} -\gamma_{n,M} \Bigr\}
\\ \leq  
\exp\Bigl\{
-e^{-x}
\cdot
\mathcal{G}_M
\cdot
\Bigl(
	1+a_M^{-2} 
	\bigl(
	-x^2(S/2) - xS +\mathfrak{c}_1/\mathfrak{c}_0
	\bigr)
	+o(1)
\Bigr)
\Bigr\} 
+  \C_1 n^{-\kappa},
\end{multline*}
as \(M \to \infty,\) where 
\begin{eqnarray}
\label{Gn} \mathcal{G}_M = \frac{M \mathfrak{c}_0}{2 \pi 
e^{a_M^2/(2 \S)}  a_M}.
\end{eqnarray}
Now we specify \(a_M\) such that \(\mathcal{G}_M \asymp 1\).
More concretely, let us find \(a_M\) in the form \(a_M = c_M - d_M /c_M\), where \(c_M, d_M\to \infty, d_M/c_M \to 0\) as \(M \to \infty\). This form leads to the equalities
\begin{eqnarray*}
M = e^{c_M^2/(2S)}, \qquad e^{d_M/S} \mathfrak{c}_0 = 2\pi c_M,
\end{eqnarray*}
which suggest
\begin{eqnarray*}
c_M = \sqrt{2S \log(M)}, \qquad d_M = S \log( 2 \pi c_M/\mathfrak{c}_0). 
\end{eqnarray*}
Under this choice, we get 
\begin{eqnarray}
\label{GGn}
\mathcal{G}_M = 1-\frac{d_M^2}{2S c_M^2} (1+o(1)).
\end{eqnarray}
Therefore, 
\begin{multline*}
\P \Bigl\{\sqrt{\frac{n}{M_n}}\RR_{n} \leq a_M + \frac{xS}{a_M} -\gamma_{n,M}\Bigr\}\\
 \leq 
\exp\Bigl\{
-e^{-x}
\cdot
\Bigl(
1-\frac{d_M^2}{2S c_M^2} (1+o(1) \Bigr) 
\cdot
\Bigl(
	1+c_M^{-2} 
	\bigl(
	-x^2(S/2) - xS +\mathfrak{c}_1/\mathfrak{c}_0
	\bigr)
	+o(1)
\Bigr)
\Bigr\} 
\\+  \C_1 n^{-\kappa}\\
= 
\exp\Bigl\{
-e^{-x} \bigl( 
1 - 
\Lambda_M (1+o(1))
\Bigr\}
+ \C_1 n^{-\kappa}
=
e^{
-e^{-x} 
}
\bigl( 
1 + e^{-x}
\Lambda_M
(1+o(1))
\bigr).
\end{multline*}
After the change \(x\) by \(x+\gamma_{n,M} a_M/S\), we get
\begin{eqnarray*}
\P \Bigl\{\sqrt{\frac{n}{M_n}}\RR_{n} \leq u_M(x)\Bigr\}
&\leq& 
e^{
-e^{-x} 
}
\bigl( 
1 + e^{-x}
\Lambda_M
(1+o(1))
\bigr),\quad n\to \infty,\end{eqnarray*}
because \(\gamma_{n,M} a_M\) converges to zero at polynomial rate (here we use the assumption \(M=n^\lambda\), \(\lambda \in (0,1)\) at the first time, see Remark~\ref{rem1}).  
The proof of the inverse inequality follows from the second statement of the Proposition~\ref{prop1}.

\underline{\textbf{(ii)}.} It holds
\begin{eqnarray}\label{y1} 
\bigl|
\hat{p}_n(x)  - p(x) 
\bigr| 
\leq 
\bigl| 
\hat{p}_n(x)  - \E \hat{p}_n(x) 
\bigr| 
+ 
\bigl| 
 \E \hat{p}_n(x)  - p(x)
\bigr|.
\end{eqnarray}
Let us show that  the second summand can be upper bounded by an expression of order \(M^{-1}\). In fact, for any \(x \in [A,B]\), 
\begin{eqnarray*}
0 &=& \sum_{m=1}^{M}\int_{I_{m}} \psim(y) dy \cdot \psim(x), \qquad \; j=1,2,...,\\
1 &=&  \sum_{m=1}^{M}\int_{I_{m}} \psi^{(m)}_0(y) dy \cdot \psi^{(m)}_0(x).
\end{eqnarray*}
Both equalities are obtained by direct calculations; e.g., the first equality  follows from 
\begin{eqnarray*}
\int_{I_{m}} \psim(y) dy &=& M^{-1/2} \sqrt{\frac{2j+1}{2}} \int_{-1}^1 \psi_j(x) dx \\&=& 
M^{-1/2} \sqrt{\frac{2j+1}{2}} \frac{1}{2^n n!} \frac{d^{n-1}}{dx^{n-1}}(x^2-1)^n |_{-1}^1 =0.
\end{eqnarray*}
Therefore, 
\begin{eqnarray*}
\E\hat{p}_n(x) - p(x) = \sum_{m=1}^M \sum_{j=0}^J 
\left[ 
\int_{I_{m}}\psim(y) 
\left( 
	p(y) - p(x)
\right) dy \cdot \psim(x) 
\right].
\end{eqnarray*}
Applying the Cauchy-Schwarz inequality for the second sum, we get
 \begin{multline}
 \label{tss}
\left| 
	\E\hat{p}_n(x) - p(x) 
\right| \leq \sum_{m=1}^M
\left(
	\sum_{j=0}^J  
	\left(
		\int_{I_{m}} \psim(y)
		\left( 
			p(y) - p(x)
		\right) dy
	\right)^{2}
\right)^{1/2}\\
\left(
	\sum_{j=0}^J 
	 (\psim(x))^{2}	
\right)^{1/2}.
\end{multline}
Next, we apply the Cauchy-Schwarz inequality  for the integral in \eqref{tss}:
 \begin{multline}
 \label{tss2}
\left| 
\E\hat{p}_n(x) - p(x) 
\right| \leq \sum_{m=1}^M 
\left(
	\int_{I_{m}} \left(
p(y)- p(x)
		\right)^{2} dy
		\cdot
	\sum_{j=0}^J  
		\int_{I_{m}} \left(\psim(y)\right)^{2} dy
\right)^{1/2}
\\
\cdot
\left(
	\sum_{j=0}^J 
	 \left( \psim(x)\right)^2
\right)^{1/2}.
\end{multline}
Now we will use  that 
\begin{eqnarray*}
\int_{I_{m}} \left(\psim(y)\right)^2 dy=1, \; \forall \; j,m, \qquad \qquad \sum_{j=0}^J 
	 \left( \psim(x)\right)^2 \leq C_1 M,
	 \end{eqnarray*}
with some \(C_1>0\) depending on \(J\). We have \begin{eqnarray*}
\int_{I_{m}} \left(
	p(y)- p(x)
\right)^{2} dy
\leq C_2 M^{-(2\beta+1)}, \qquad \forall \; x \in I_{m}.
\end{eqnarray*}
We arrive at  
\begin{eqnarray*}
\left| 
	\E \hat{p}_n(x) - p(x) 
\right| \leq C_3  M^{-\beta}, \qquad \forall \; x \in [A,B].\end{eqnarray*}
with some constant \(C_3>0\). Substituting this result into~\eqref{y1}, we get 
\begin{eqnarray*} 
		\sqrt{\frac{n}{M}}
		\DDD_{n}
		 \leq			\sqrt{\frac{n}{M}}
\RR_{n}		+	
C_4 n^{1/2} M^{-\beta-1/2} 
\end{eqnarray*}
where \( C_4  = C_3 q^{-1}\).  On the other side, we have 
\begin{eqnarray*}
	\left| 
		\hat{p}_n (x)  - p(x)
	\right|
	\geq
	\left| 
		 \hat{p}_n (x)  - \E \hat{p}_n (x)
	\right|
	-
	\left| 
		\E \hat{p}_n (x)
 - p(x)
	\right|,
\end{eqnarray*}
and therefore
\begin{eqnarray*} 
		\sqrt{\frac{n}{M}}
		\DDD_{n}	
		 \geq
		 \sqrt{\frac{n}{M}}
\RR_{n}		-	
C_4 n^{1/2} M^{-\beta-1/2}.
\end{eqnarray*}
We conclude that
\begin{eqnarray*}
\P \left\{
\left|
	\sqrt{\frac{n}{M}}
	\DDD_{n}		
	-
	\sqrt{\frac{n}{M}}
	\RR_{n}	
\right| 
\leq C_4 n^{1/2} M^{-\beta-1/2}
\right\}=1.
\end{eqnarray*}
By Lemma~\ref{final},  for any \(x \in \R,\)
\begin{multline}
\label{ineq}
\P \left\{
	\sqrt{\frac{n}{M}}
	\RR_{n}
	\leq x 
	-
	C_4 n^{1/2} M^{-\beta-1/2}
\right\}
\leq 
\P \left\{
	\sqrt{\frac{n}{M}}
	\DDD_{n}
	\leq x
\right\}
\\ 
\leq 
\P \left\{
\sqrt{\frac{n}{M}}
\RR_{n}	\leq x 	+	
C_4 n^{1/2} M^{-\beta-1/2}
\right\}.
\end{multline}
Substituting \(u_{M}(x)\) defined by \eqref{un}-\eqref{un2} instead of \(x,\) we get that the  left-hand side in \eqref{ineq} can be transformed as follows  
\begin{multline*}
\P \left\{
	\sqrt{\frac{n}{M}}
	\RR_{n}
	\leq u_M(x)
	-
	C_2 n^{1/2} M^{-\beta-1/2}
\right\}
\\ 
=
\P \left\{
	\sqrt{\frac{n}{M}}
	\RR_{n}
	\leq u_{M}\bigl(
		x-C_2 n^{1/2} M^{-\beta-1/2}a_{M}/S
	\bigr)
\right\}\\
= e^{-  e^{-x}} \left( 
			1 +e^{-x}  \Lambda_M (1+o(1))
		\right),
\end{multline*}
provided \(n^{1/2} M^{-\beta-1/2}a_{M}\) converges to 0 at polynomial rate. The last condition is fulfilled for any \(\alpha \in (1/(2\beta+1),1).\)
The same argument holds for the right-hand side of \eqref{ineq}, and the desired result follows.

\end{proof}

\section{One technical lemma}\label{secT}
\begin{customthm}{T}\label{lemapp} Denote 
\[ \Theta_{n,M}(x) :=  A_M(x+w_{n,M}) - A_M(x),\]
where \(w_{n,M}>0\) converges to zero as \(n,M \to \infty\). Then 
\begin{eqnarray*}
\sup_{x \in \R} \Theta_{n,M}(x) \leq c_1  M^{\theta_1} w_{n,M} + c_2 M^{-\theta_2}
\end{eqnarray*}\end{customthm}
for some \(c_1,c_2>0\) and any \(\theta_1, \theta_2>0.\)
\begin{proof}

For \(x<c_M-w_{n,M}\), we have \( \Theta_{n,M}(x) = 0\).

If \(x\geq c_M\), 
\begin{eqnarray*}
\Theta_{n,M}(x) &\leq& \sup_{x\geq c_M} \bigl[ A_M'(x)  \bigr] w_{n,M}=
\sup_{x\geq c_M}  \Bigl[ - 
 A_M(x)  \sum_{i=1}^k \pp'_i\bigl( 
x
\bigr) 
\Bigr]  M w_{n,M}.
\end{eqnarray*}
Note that for \(x \geq c_M\), we have 
\begin{eqnarray*}
0 \leq -
\sum_{i=1}^k \pp'_i\bigl( 
x
\bigr)  = \frac{2k}{\sqrt{S}} \phi\Bigl(\frac{x}{\sqrt{S}}\Bigr)
< \frac{2k}{\sqrt{2\pi S}} e^{-c_M^2/(2S)}\lesssim M^{-1} e^{\sqrt{2S\log(M)}}.
\end{eqnarray*}
Since \(e^{\sqrt{2S\log(M)}}\lesssim M^{\theta_1}\) for any \(\theta_1>0,\)
we conclude that 
\begin{eqnarray}\label{q1}
\sup_{x\geq c_M} \Theta_{n,M}(x) \lesssim 
c_1  M^{\theta_1} w_{n,M}, \qquad n,M \to \infty
\end{eqnarray}
for any \(\theta_1>0.\) Finally, if \(x \in (c_M-w_{n,M}, c_M)\), we have
\begin{eqnarray*}
 \Theta_{n,M}(x) = A_M(x+ w_{n,M})
 <  \Theta_{n,M}(c_M) + A_M(c_M),
\end{eqnarray*} 
where the first term is bounded by the expression in the right-hand side of \eqref{q1}.
Now let us consider the second term: 
\begin{eqnarray*}
A_M(c_M) &=& 
\exp \Bigl\{ 
-\frac{M \mathfrak{c}_0}{ 2\pi c_M}
e^{-c_M^2 / (2S)}
(1 +o(1))
\Bigr\}\\
&=& 
\exp \Bigl\{ 
-\frac{ \mathfrak{c}_0 e^{-S/2}}{ 2\pi}
\frac{ e^{(2S \log M)^{1/2}}}{(2S \log(M))^{1/2}}
(1 +o(1))
\Bigr\}.
\end{eqnarray*}
Applying~\eqref{bol}, we conclude that \(A_M(c_M)\) converges to zero at polynomial rate with respect to \(M\). This observation completes the proof of Lemma~\ref{lemapp}.\end{proof}

\end{document}